\documentclass[reqno]{amsart}
\usepackage{amsmath}
\usepackage{graphicx}
\usepackage{array}
\usepackage{epsfig}
\usepackage{float,array}
\usepackage{amsfonts}
\usepackage{amssymb}
\usepackage[all,cmtip]{xy}
\usepackage{xcolor}
\usepackage[colorlinks=true]{hyperref}

\pdfoutput=1
\hypersetup{linkcolor={red!80!black},citecolor={blue!80!black},urlcolor={blue!80!black}}
\newtheorem{theorem}{Theorem}[section]
\newtheorem{corollary}[theorem]{Corollary}

\newtheorem{lemma}[theorem]{Lemma}
\newtheorem{proposition}[theorem]{Proposition}

\theoremstyle{definition}
\newtheorem{definition}[theorem]{Definition}

\newtheorem{examples}[theorem]{Examples}
\newtheorem{remark}[theorem]{Remark}
\newtheorem{note}[theorem]{Note}
\numberwithin{equation}{section}
\input{tcilatex}

\begin{document}
\author[D.I. Dais]{Dimitrios I. Dais}
\address{University of Crete, Department of Mathematics and Applied
Mathematics, Division Algebra and Geometry, Voutes Campus, P.O. Box 2208,
GR-70013, Heraklion, Crete, Greece}
\email{ddais@math.uoc.gr}
\title[Toric log Del Pezzo surfaces with one singularity]{Toric log Del
Pezzo Surfaces \\
with One Singularity}
\date{}

\begin{abstract}
This paper focuses on the classification of all toric log Del Pezzo surfaces
with exactly one singularity up to isomorphism, and on the description of
how they are embedded as intersections of finitely many quadrics into
suitable projective spaces.
\end{abstract}

\maketitle

\section{Introduction\label{INTRO}}

\noindent A smooth compact complex surface $X$ is called \textit{del Pezzo} 
\textit{surface} if its anticanonical divisor $-K_{X}$ is ample, i.e., if
the rational map $\Phi_{\left\vert -mK_{X}\right\vert }:X\dashrightarrow 
\mathbb{P}(\left\vert -mK_{X}\right\vert )$ associated to the linear system $%
\left\vert -mK_{X}\right\vert $ becomes a closed embedding with 
\begin{equation*}
\mathcal{O}_{X}(-mK_{X})\cong\Phi_{\left\vert -mK_{X}\right\vert }^{\ast
}\left( \mathcal{O}_{\mathbb{P}(\left\vert -mK_{X}\right\vert )}\left(
1\right) \right) ,
\end{equation*}
for a suitable positive integer $m.$ (Pasquale del Pezzo \cite{del Pezzo}
initiated the study of these surfaces in $1887.$) The \textit{degree} $%
\deg(X)$ of a del Pezzo surface $X$ is defined to be the self-intersection
number $(-K_{X})^{2}.$ The main classification result about these surfaces
can be stated as follows (see \cite[Theorem 24.4, pp. 119-121]{Manin}):

\begin{theorem}
\label{CLDELPS}Let $X$ be a del Pezzo surface of degree $d:=\deg(X).$ We
have necessarily $1\leq d\leq9,$ and $X$ is classified by $d$\emph{%
:\smallskip \newline
(i)} If $d=9,$ then $X$ is isomorphic to the projective plane $\mathbb{P}_{%
\mathbb{C}}^{2}.\smallskip$\newline
\emph{(ii)} If $d=8,$ then $X$ is isomorphic either to $\mathbb{P}_{\mathbb{C%
}}^{1}\times\mathbb{P}_{\mathbb{C}}^{1}$ or to the blow-up of the projective
plane $\mathbb{P}_{\mathbb{C}}^{2}$ at one point.\smallskip\ \newline
\emph{(iii)} If $1\leq d\leq7,$ then $X$ is isomorphic to the blow-up of the
projective plane $\mathbb{P}_{\mathbb{C}}^{2}$ at $9-d$ points$.$
\end{theorem}

\noindent{}For $6\leq d\leq9,$ such an $X$ is \textit{toric}, i.e., it
contains a $2$-dimensional algebraic torus $\mathbb{T}$ as a dense open
subset, and is equipped with an algebraic action of $\mathbb{T}$ on $X$
which extends the natural action of $\mathbb{T}$ on itself. Taking into
account the description of smooth compact toric surfaces by the ($\mathbb{Z}$%
-weighted) circular graphs (introduced in \cite[Chapter I, \S 8]{Oda-Miyake}%
, \cite[pp. 42-46]{Oda}), as well as \cite[Proposition 6]{Batyrev} and \cite[%
Proposition 2.7]{WW}, Oda expresses in \cite[Proposition 2.21, pp. 88-89]%
{Oda} this fact in the language of toric geometry as follows:

\begin{theorem}
\label{CLTORDP}There exist five distinct toric del Pezzo surfaces up to
isomorphism. They correspond to the circular graphs \emph{(}with weights $%
-1,0,1$\emph{)} shown in Figure \emph{\ref{Fig.1}}. They are \emph{(i)} $%
\mathbb{P}_{\mathbb{C}}^{2},$ \emph{(ii)} $\mathbb{P}_{\mathbb{C}}^{1}\times%
\mathbb{P}_{\mathbb{C}}^{1}$ \emph{(}$\cong\mathbb{F}_{0}$\emph{)}$,$ \emph{%
(iii)} the Hirzebruch surface $\mathbb{F}_{1},$ \emph{(iv)} the equivariant
blow-up of $\ \mathbb{P}_{\mathbb{C}}^{2}$ at two of the $\mathbb{T}$-fixed
points, and \emph{(v)} the equivariant blow-up of $\mathbb{P}_{\mathbb{C}%
}^{2}$ at the three $\mathbb{T}$-fixed points.
\end{theorem}

\begin{figure}[ht]
\includegraphics[height=3cm, width=11cm]{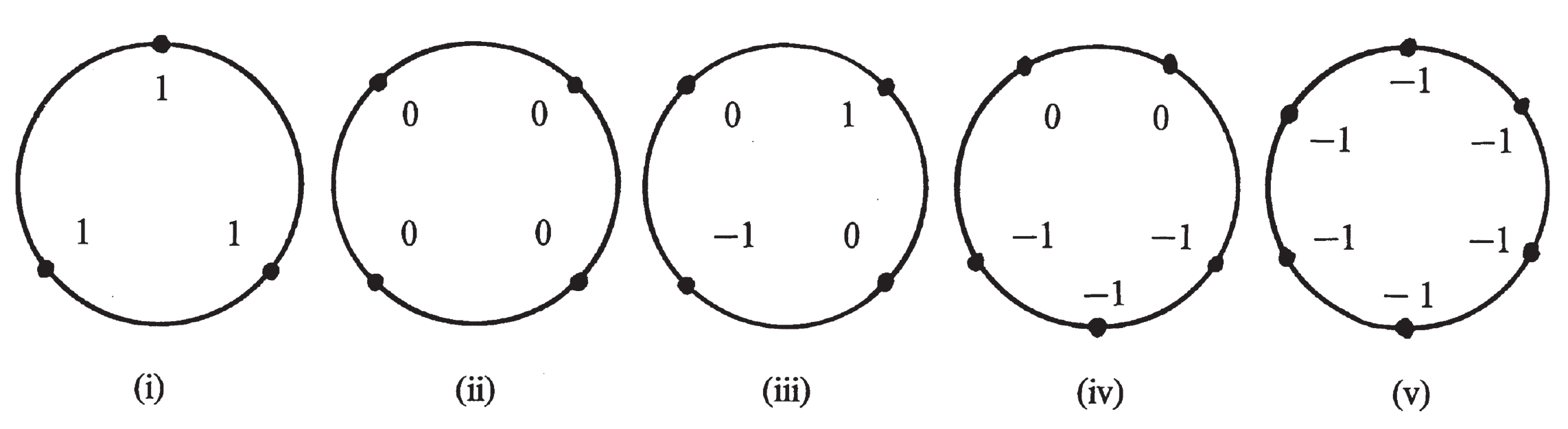}
\caption{}
\label{Fig.1}
\end{figure}

\begin{note}
The so-called \textit{Hirzebruch surfaces }(introduced in \cite[\S 2]%
{Hirzebruch})%
\begin{equation*}
\mathbb{F}_{\kappa}:=\left\{ \left. (\left[ z_{0}:z_{1}:z_{2}\right] ,\left[
t_{1}:t_{2}\right] )\in\mathbb{P}_{\mathbb{C}}^{2}\times \mathbb{P}_{\mathbb{%
C}}^{1}\ \right\vert \ z_{1}t_{1}^{\kappa}=z_{2}t_{2}^{\kappa}\right\} ,\
\kappa\in\mathbb{Z}_{\geq0},
\end{equation*}
are toric. $\mathbb{F}_{\kappa}$ is usually identified with the total space $%
\mathbb{P}(\mathcal{O}_{\mathbb{P}_{\mathbb{C}}^{1}}\oplus\mathcal{O}_{%
\mathbb{P}_{\mathbb{C}}^{1}}(\kappa))$\ of the $\mathbb{P}_{\mathbb{C}}^{1}$%
-bundle of degree $\kappa$ over $\mathbb{P}_{\mathbb{C}}^{1}.$ Furthermore,
every smooth compact toric surface which has Picard number $2$ is
necessarily isomorphic to a Hirzebruch surface (cf. \cite[Corollary 1.29, p.
45]{Oda}).
\end{note}

\noindent{}$\blacktriangleright$ \textit{The singular analogues}. A normal
compact complex surface $X$ with at worst log terminal singularities, i.e.,
quotient singularities, is called \textit{log del Pezzo} \textit{surface} if
its anticanonical Weil divisor $-K_{X}$ is a $\mathbb{Q}$-Cartier ample
divisor. The \textit{index} of such an $X$ is defined to be the smallest
positive integer $\ell$ for which $-\ell K_{X}$ is a Cartier divisor. The
family of log del Pezzo surfaces of fixed index $\ell$ is known to be
bounded. (See Nikulin \cite{Nikulin1}, \cite{Nikulin2}, \cite{Nikulin3}, and
Borisov \cite[Theorem 2.1, p. 332]{Borisov}.) Consequently, it seems to be
rather interesting to classify log del Pezzo surfaces of given index $\ell$.
This has been done for index $\ell=1$ by Hidaka \& Watanabe \cite{H-W} (by a
direct generalization of Theorem \ref{CLDELPS}) and Ye \cite{Ye}, and for
index $\ell=2$ by Alexeev \& Nikulin \cite{A-N1}, \cite{A-N2} (in terms of
diagrams of exceptional curves w.r.t. a suitable resolution of
singularities). Related results are due to Kojima \cite{Kojima2} (whenever
the Picard number equals $1$) and Nakayama \cite{Nakayama} (whose techniques
apply even if one replaces $\mathbb{C}$ with an algebraically closed field
of arbitrary characteristic). Based on Nakayama's arguments, Fujita \&
Yasutake \cite{F-Y} succeeded recently to extend the classification even for 
$\ell=3.$ But for indices $\ell\geq4$ the situation turns out to be much
more complicated, and (apart from some partial results as those in \cite%
{Fujita1}, \cite{Fujita2}) it is hard to expect a complete characterization
of these surfaces in this degree of generality.

On the other hand, if we restrict our study to the subclass of \textit{toric}
log del Pezzo surfaces, the classification problem becomes considerably
simpler: a) The only singularities which can occur are \textit{cyclic}
quotient singularities. b) To classify (not necessarily smooth) compact
toric surfaces up to isomorphism it is enough to use the graph-theoretic
method proposed in \cite[\S 5]{Dais1} (which generalizes Oda's graphs
mentioned above): Two compact toric surfaces are isomorphic to each other if
and only if their vertex singly- and edge doubly-weighted circular graphs (%
\textsc{wve}$^{2}$\textsc{c}-\textit{graphs}, for short) are isomorphic (see
below Theorem \ref{CLASSIFTHM}). A detailed examination of the
number-theoretic properties of the weights of these graphs led to the
classification of all toric log del Pezzo surfaces having Picard number $1$
and index $\ell\leq3$ in \cite[\S 6]{Dais1} and \cite{Dais2}. In fact, the
purely combinatorial part of the classification problem can be further
simplified because it can be reduced to the classification of the so-called 
\textit{LDP-polygons} (introduced in \cite{Dais-Nill}) \textit{up to
unimodular transformation}. For $\ell=1$ these are the sixteen \textit{%
reflexive polygons} (which were discovered by Batyrev in the 1980's). More
recently, Kasprzyk, Kreuzer \& Nill \cite[\S 6]{KKN} developed a particular
algorithm by means of which one creates an LDP-polygon (for given $\ell\geq2$%
) by fixing a \textquotedblleft special\textquotedblright\ edge and
following a prescribed successive addition of vertices, and produced in this
way the long lists of \textit{all} LDP-polygons for $\ell\leq17.$ (An
explicit study for each of these $15346$ LDP-polygons is available on the
webpage \cite{Br-Kas}.)$\smallskip$

\noindent{}$\blacktriangleright$ \textit{Restrictions on the singularities}.
At this point let us mention some remarkable results concerning the
singularities of log del Pezzo surfaces having Picard number $1$: Belousov
proved in \cite{Belousov1}, \cite{Belousov2} that each of these surfaces
admits at most $4$ singularities, Kojima \cite{Kojima1} described the nature
of the exceptional divisors w.r.t. the minimal resolution of those
possessing exactly one singularity, and Elagin \cite{Elagin} constructed
certain (non-toric) surfaces of this kind (realized as hypersurfaces of
degree $4n-2$ in $\mathbb{P}_{\mathbb{C}}^{3}(1,2,2n-1,4n-3)$), and proved
the existence of full exceptional sets of coherent sheaves over them.

Obviously, the maximal number of the singularities of a \textit{toric} log
del Pezzo surface equals the number of the edges of the corresponding
LDP-polygon. (For an upper bound of this number see \cite[Lemma 3.1]%
{Dais-Nill}.) In the present paper we classify all toric log del Pezzo
surfaces \textit{with exactly one singularity} (without laying a priori any
restrictions on the Picard number or on the index) \textit{up to isomorphism}%
.

\begin{theorem}
\label{MAIN1}Let $X_{Q}$ be a toric log del Pezzo surface \emph{(}associated
to an LDP-polygon $Q$\emph{)} with exactly one singularity. Then the
following hold true\emph{:\smallskip} \newline
\emph{(i)} The Picard number $\rho\left( X_{Q}\right) $ of $X_{Q}$ can take
only the values $1,2$ and $3.\smallskip$ \newline
\emph{(ii) }If we define for every positive integer $p$ the LDP-polygons%
\begin{equation}
\left\{ 
\begin{array}{l}
Q_{p}^{\left[ 1\right] }:=\text{\emph{conv}}\left( \left\{ \tbinom{1}{-1},%
\tbinom{p}{1},\tbinom{-1}{0}\right\} \right) ,\smallskip \\ 
Q_{p}^{\left[ 2\right] }:=\text{\emph{conv}}\left( \left\{ \tbinom{1}{-1},%
\tbinom{p}{1},\tbinom{p-1}{1},\tbinom{-1}{0}\right\} \right) ,\smallskip \\ 
Q_{p}^{\left[ 3\right] }:=\text{\emph{conv}}\left( \left\{ \tbinom{1}{-1},%
\tbinom{p}{1},\tbinom{p-1}{1},\tbinom{-1}{0},\tbinom{0}{-1}\right\} \right)%
\end{array}
\right\} ,  \label{DEFQS}
\end{equation}
then for $k\in\{1,2,3\}$ we have%
\begin{equation*}
\rho\left( X_{Q}\right) =k\Longleftrightarrow\exists p\in\mathbb{Z}%
_{>0}:X_{Q}\cong X_{Q_{p}^{\left[ k\right] }},
\end{equation*}
and the \textsc{wve}$^{2}$\textsc{c}-graphs $\mathfrak{G}_{\Delta _{Q_{p}^{%
\left[ k\right] }}}$ are those depicted in \emph{Figure \ref{Fig.2}}%
.\smallskip\ \newline
\emph{(iii)} $X_{Q_{p}^{\left[ 1\right] }}$ is isomorphic to the weighted
projective plane $\mathbb{P}_{\mathbb{C}}^{2}(1,1,p+1)$ and is obtained by
contracting the $\infty$-section $\mathbb{P}(\mathcal{O}_{\mathbb{P}_{%
\mathbb{C}}^{1}}(p+1))$ of $\ \mathbb{F}_{p+1}$. The surface $X_{Q_{p}^{%
\left[ 2\right] }}$ is obtained by blowing up a Hirzebruch surface $\mathbb{F%
}_{p}$ at one $\mathbb{T}$-fixed point, and contracting afterwards its $%
\infty$-section. $X_{Q_{p}^{\left[ 3\right] }}$ is obtained by blowing up $%
X_{Q_{p}^{\left[ 2\right] }}$ at one non-singular $\mathbb{T}$-fixed
point.\smallskip\ \newline
\emph{(iv)} If $X_{Q}$ has index $\ell\geq1$ and Picard number $\rho\left(
X_{Q}\right) =k\in\{1,2,3\},$ then for $\ell$ odd $\geq3$ either $X_{Q}\cong
X_{Q_{\ell-1}^{\left[ k\right] }}$ or $X_{Q}\cong X_{Q_{2\ell-1}^{\left[ k%
\right] }},$ whereas for $\ell \in\{1\}\cup2\mathbb{Z}$ we have $X_{Q}\cong
X_{Q_{2\ell-1}^{\left[ k\right] }}.$
\end{theorem}

\begin{figure}[ht]
\includegraphics[height=3.9cm, width=11cm]{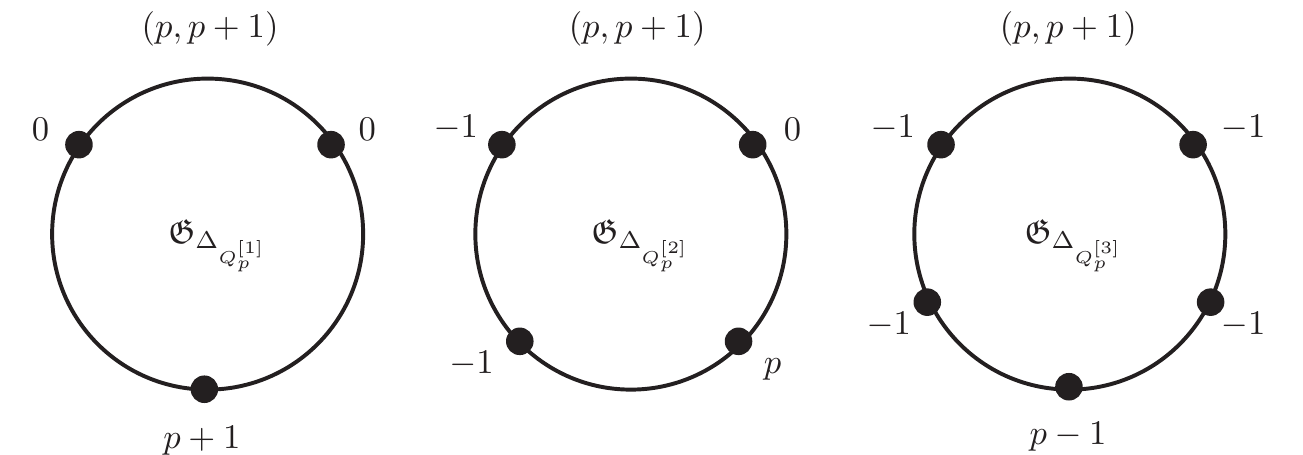}
\caption{}
\label{Fig.2}
\end{figure}

\noindent{}$\blacktriangleright$ \textit{Equations defining closed
embeddings.} For every del Pezzo surface $X$ of degree $d$ with $3\leq
d\leq9 $ the anticanonical divisor $-K_{X}$ is already very ample, and $\Phi
_{\left\vert -K_{X}\right\vert }$ gives rise to a realization of $X$ as a
subvariety of projective degree $d$ in $\mathbb{P}_{\mathbb{C}}^{d}.$ (For $%
d=1$ and $d=2,$ one has to work with $-3K_{X}$ and $-2K_{X}$ instead to
obtain realizations of $X$ as a subvariety of degree $9,$ and of degree $8$
in $\mathbb{P}_{\mathbb{C}}^{6},$ respectively.) Generalizations of these
(or similar but more \textquotedblleft economic\textquotedblright)
embeddings of \textit{log} del Pezzo surfaces of index $1$ and $2$ (in
appropriate projective or weighted projective spaces) appear in \cite{H-W}
and \cite{Ka-Ka}. Since every ample divisor on a compact \textit{toric}
surface is very ample (cf. \cite{E-W} or \cite[Corollary 2.2.19 (b), p. 71,
and Proposition 6.1.10, pp. 269-270 ]{CLS}), the map $\Phi_{\left\vert -\ell
K_{X_{Q}}\right\vert }$ associated to the linear system $\left\vert -\ell
K_{X_{Q}}\right\vert $ on a toric log del Pezzo surface $X_{Q}$ of index $%
\ell$ becomes a closed embedding. Koelman's Theorem \cite{Koelman} and
standard lattice point enumeration techniques enable us to describe $%
\Phi_{\left\vert -\ell K_{X_{Q}}\right\vert }(X_{Q})$ for those $X_{Q}$'s
classified in Theorem \ref{MAIN1} as follows: \renewcommand{%
\arraystretch}{1.9}

\begin{theorem}
\label{MAIN2}Let $X_{Q}$ be a toric log del Pezzo surface of index $%
\ell\geq1 $ with exactly one singularity. Then the image of $X_{Q}\cong
X_{Q_{p}^{\left[ k\right] }}$ under the closed embedding 
\begin{equation*}
\Phi_{\left\vert -\ell K_{X_{Q}}\right\vert }:X_{Q}\hookrightarrow \mathbb{P}%
(\left\vert -\ell K_{X_{Q}}\right\vert )
\end{equation*}
is isomorphic to a subvariety of $\mathbb{P}_{\mathbb{C}}^{\delta _{Q_{p}^{%
\left[ k\right] }}}$ of projective degree $d_{Q_{p}^{\left[ k\right] }}$
which can be expressed as intersection of finitely many quadrics; $%
\delta_{Q_{p}^{\left[ k\right] }}$ and $d_{Q_{p}^{\left[ k\right] }}$ are
given in the table\emph{:}{\small 
\begin{equation}
\begin{tabular}{|c|c|c|c|c|}
\hline
\emph{No.} & $p$ & $k$ & $%
\begin{array}{c}
d_{Q_{p}^{\left[ k\right]}} \vspace{0.1cm} \\ 
\end{array}%
$ & $\delta _{Q_{p}^{\left[ k\right] }}$ \\ \hline\hline
\emph{(i)} & \emph{odd} & $1$ & $\frac{1}{4}\left( p+1\right) \left(
p+3\right) ^{2}$ & $\frac{1}{8}\left( p+3\right) ^{3}$ \\ \hline
\emph{(ii)} & \emph{even} & $1$ & $\left( p+1\right) \left( p+3\right) ^{2}$
& $\frac{1}{2}\left( p+2\right) \left( p+3\right) ^{2}$ \\ \hline
\emph{(iii)} & \emph{odd} & $2$ & $\frac{1}{4}\allowbreak\left( p+1\right)
\left( p^{2}+5p+8\right) $ & $\frac{1}{8}\left( p+3\right) \left(
p^{2}+5p+8\right) $ \\ \hline
\emph{(iv)} & \emph{even} & $2$ & $\left( p+1\right) \left(
p^{2}+5p+8\right) $ & $\frac{1}{2}\left( p+2\right) \left( p^{2}+5p+8\right) 
$ \\ \hline
\emph{(v)} & \emph{odd} & $3$ & $\frac{1}{4}\left( p+1\right) \left(
p^{2}+4p+7\right) $ & $\frac{1}{8}\left( p+3\right) \left( p^{2}+4p+7\right) 
$ \\ \hline
\emph{(vi)} & \emph{even} & $3$ & $\left( p+1\right) \left(
p^{2}+4p+7\right) $ & $\frac{1}{2}\left( p+2\right) \left( p^{2}+4p+7\right) 
$ \\ \hline
\end{tabular}
\ \ \ \ \ \ \   \label{TABELLE1}
\end{equation}
} On the other hand, the cardinality $\beta_{Q_{p}^{\left[ k\right] }}$ of
any minimal system of quadrics (generating the ideal which determines this
subvariety) equals{\small 
\begin{equation}
\begin{tabular}{|c|c|c|c|}
\hline
\emph{No.} & $p$ & $k$ & $%
\begin{array}{c}
\beta_{Q_{p}^{\left[ k\right] }} \vspace{0.1cm} \\ 
\end{array}%
$ \\ \hline\hline
\emph{(i)} & \emph{odd} & $1$ & $\frac{1}{128}\left( p+1\right) \left(
p+3\right) ^{2}\left( p^{3}+11p^{2}+43p+25\right) \allowbreak$ \\ \hline
\emph{(ii)} & \emph{even} & $1$ & $\frac{1}{8}\left( p+3\right) ^{2}\left(
p^{4}+10p^{3}+37p^{2}+50p+24\right) $ \\ \hline
\emph{(iii)} & \emph{odd} & $2$ & $\frac{1}{128}\left( p+1\right) \left(
p^{2}+5p+8\right) \left( p^{3}+10p^{2}+37p+16\right) $ \\ \hline
\emph{(iv)} & \emph{even} & $2$ & $\frac{1}{8}\left( p^{2}+5p+8\right)
\left( p^{4}+9p^{3}+32p^{2}+42p+20\right) $ \\ \hline
\emph{(v)} & \emph{odd} & $3$ & $\frac{1}{128}\left( p+1\right) \left(
p^{2}+4p+7\right) \left( p^{3}+9p^{2}+31p+7\right) $ \\ \hline
\emph{(vi)} & \emph{even} & $3$ & $\frac{1}{8}\left( p^{2}+4p+7\right)
\left( p^{4}+8p^{3}+27p^{2}+34p+16\right) $ \\ \hline
\end{tabular}
\ \ \ \ \ \   \label{TABELLE2}
\end{equation}
} and the sectional genus $g_{Q_{p}^{\left[ k\right] }}\emph{\medskip}$ of $%
X_{Q_{p}^{\left[ k\right] }}$ equals{\small 
\begin{equation}
\begin{tabular}{|c|c|c|c|}
\hline
\emph{No.} & $p$ & $k$ & $%
\begin{array}{c}
g_{Q_{p}^{\left[ k\right] }} \vspace{0.1cm} \\ 
\end{array}%
$ \\ \hline\hline
\emph{(i)} & \emph{odd} & $1$ & $\frac{1}{8}\left( p+1\right) \left(
p^{2}+4p-1\right) $ \\ \hline
\emph{(ii)} & \emph{even} & $1$ & $\frac{1}{2}\left( p+2\right) \left(
p^{2}+4p-1\right) $ \\ \hline
\emph{(iii)} & \emph{odd} & $2$ & $\frac{1}{8}p\left( p+1\right) \left(
p+3\right) $ \\ \hline
\emph{(iv)} & \emph{even} & $2$ & $\frac{1}{2}\left(
p^{3}+5p^{2}+8p+2\right) $ \\ \hline
\emph{(v)} & \emph{odd} & $3$ & $\frac{1}{8}\left( p+1\right) ^{3}$ \\ \hline
\emph{(vi)} & \emph{even} & $3$ & $\frac{1}{2}\left(
p^{3}+4p^{2}+7p+2\right) $ \\ \hline
\end{tabular}
\ \ \ \ \   \label{TABELLE3}
\end{equation}
}
\end{theorem}

\noindent{}The paper is organized as follows: In \S \ref{2DIMTORSING} we
focus on the two non-negative, relatively prime integers $p=p_{\sigma}$ and $%
q=q_{\sigma}$ parametrizing the $2$-dimensional, rational, strongly convex
polyhedral cones $\sigma,$ and explain how they characterize the $2$%
-dimensional toric singularities. In \S \ref{COMTORS}-\S \ref{LDPPOL} we
recall some auxiliary geometric properties of compact toric surfaces and of
those which are log del Pezzo. The proofs of Theorems \ref{MAIN1} and \ref%
{MAIN2} are given in sections \ref{PROOFOFMAIN1} and \ref{PROOFOFMAIN2},
respectively. We use tools only from discrete and classical toric geometry,
adopting the standard terminology from \cite{CLS}, \cite{Ewald}, \cite%
{Fulton}, and \cite{Oda} (and mostly the notation introduced in \cite{Dais1}%
).

\section{Two-dimensional toric singularities\label{2DIMTORSING}}

\noindent{}Let $\sigma=\mathbb{R}_{\geq0}\mathbf{n}+\mathbb{R}_{\geq 0}%
\mathbf{n}^{\prime}\subset\mathbb{R}^{2}$ be a 2-dimensional, rational,
strongly convex polyhedral cone. Without loss of generality we may assume
that $\mathbf{n}=\tbinom{a}{b}$, $\mathbf{n}^{\prime}=\tbinom{c}{d}\in 
\mathbb{Z}^{2},$ and that both $\mathbf{n}$ and $\mathbf{n}^{\prime}$ are
primitive elements of $\mathbb{Z}^{2}$, i.e., gcd$\left( a,b\right) =1$ and
gcd$\left( c,d\right) =1.$

\begin{lemma}
\label{pequ}Consider $\kappa,\lambda\in\mathbb{Z},$ such that $\kappa
a-\lambda b=1.$ If $q:=\left\vert ad-bc\right\vert ,$ and $p$ is the unique
integer with 
\begin{equation*}
0\leq p<q\text{ \ \ \ \emph{and} \ \ \ }\kappa c-\lambda d\text{ }\equiv
p\left( \text{\emph{mod} }q\right) ,
\end{equation*}
then\emph{\ gcd}$\left( p,q\right) =1$, and there exists a primitive element 
$\mathbf{n}^{\prime\prime}\in\mathbb{Z}^{2}$ such that\ 
\begin{equation*}
\mathbf{n}^{\prime}=p\mathbf{n}+q\mathbf{n}^{\prime\prime}\text{ \emph{and}
\ }\left\{ \mathbf{n},\mathbf{n}^{\prime\prime}\right\} \text{ \emph{is a} }%
\mathbb{Z}\text{\emph{-basis of} }\mathbb{Z}^{2}.
\end{equation*}
Moreover, there is a unimodular transformation $\Psi:\mathbb{R}%
^{2}\rightarrow\mathbb{R}^{2},$ $\Psi\left( \mathbf{x}\right) :=\Xi \,%
\mathbf{x,}$ with $\Xi\in$ \emph{GL}$_{2}(\mathbb{Z}),$ such that%
\begin{equation*}
\Psi\left( \sigma\right) =\mathbb{R}_{\geq0}\tbinom{1}{0}+\mathbb{R}_{\geq 0}%
\tbinom{p}{q}.
\end{equation*}
\end{lemma}

\begin{proof}
See \cite[Lemma 2.1 and Lemma 2.2]{Dais2}.
\end{proof}

\noindent{}Henceforth, we call $\sigma$ a $(p,q)$-\textit{cone}. By $%
U_{\sigma}:=$ Spec$(\mathbb{C}[\sigma^{\vee}\cap\mathbb{Z}^{2}])$ we denote
the affine toric variety associated to $\sigma$ (by means of the monoid $%
\sigma^{\vee}\cap\mathbb{Z}^{2}$, where $\sigma^{\vee}$ is the dual of $%
\sigma$) and by orb$(\sigma)$ the single point being fixed under the usual
action of the algebraic torus $\mathbb{T}:=$ Hom$_{\mathbb{Z}}(\mathbb{Z}%
^{2},\mathbb{C}^{\ast})$ on $U_{\sigma}.$

\begin{proposition}
\label{KEY-PROP}The following conditions are equivalent:\smallskip \ \newline
\emph{(i)} $\left\{ \mathbf{n},\mathbf{n}^{\prime}\right\} $ is a $\mathbb{Z}
$-basis of $\mathbb{Z}^{2}$.\smallskip\ \ \newline
\emph{(ii)} $q=1$ \emph{(}and consequently, $p=0$\emph{).\smallskip} \newline
\emph{(iii) conv}$(\{\mathbf{0},\mathbf{n},\mathbf{n}^{\prime}\})\cap\mathbb{%
Z}^{2}=\{\mathbf{0},\mathbf{n},\mathbf{n}^{\prime}\}.$ \emph{%
(\textquotedblleft conv\textquotedblright} \emph{is abbreviation for convex
hull.)\smallskip \newline
(iv)} $U_{\sigma}\cong\mathbb{C}^{2}.$
\end{proposition}

\begin{proof}
Let $T$ be the triangle conv$(\{\mathbf{0},\mathbf{n},\mathbf{n}%
^{\prime}\}). $ The implication (i)$\Rightarrow$(ii) is obvious because%
\begin{equation*}
q=\left\vert \det(\mathbf{n},\mathbf{n}^{\prime})\right\vert =2\text{ area}%
(T).
\end{equation*}
By Pick's formula (cf. \cite[p. 113]{Fulton}) we obtain%
\begin{equation*}
\frac{q}{2}=\text{area}(T)=\sharp(\text{int}(T)\cap\mathbb{Z}^{2})+\frac{1}{2%
}\sharp(\partial(T)\cap\mathbb{Z}^{2})-1,
\end{equation*}
where \textquotedblleft int\textquotedblright\ and $\partial$ are
abbreviations for interior and boundary, respectively. If $q=1,$ then%
\begin{equation*}
\sharp(\partial(T)\cap\mathbb{Z}^{2})\geq3\Rightarrow\sharp(\text{int}(T)\cap%
\mathbb{Z}^{2})=0\text{ and necessarily }\sharp(\partial(T)\cap \mathbb{Z}%
^{2})=3.
\end{equation*}
Hence, (ii)$\Rightarrow$(iii) is also true. (iii)$\Rightarrow$(i) follows
from \cite[Theorem 4, p. 20]{G-L}. For the proof of the equivalence of
conditions (i) and (iv) see \cite[Theorem 1.10, p. 15]{Oda}.
\end{proof}

\noindent{}If the conditions of Proposition \ref{KEY-PROP} are satisfied,
then $\sigma$ is said to be a \textit{basic cone}. On the other hand,
whenever $q>1$ we have the following:

\begin{proposition}
\label{Sing-Prop}$\emph{orb}(\sigma)\in U_{\sigma}$ is a cyclic quotient
singularity. In particular, 
\begin{equation*}
U_{\sigma}\cong\mathbb{C}^{2}/G=\emph{Spec}(\mathbb{C}[z_{1},z_{2}]^{G}),
\end{equation*}
with $G\subset$ \emph{GL}$\left( 2,\mathbb{C}\right) $ denoting the cyclic
group $G$ of order $q$ which is generated by \emph{diag}$(\zeta_{q}^{-p},%
\zeta_{q})$ \emph{(}$\zeta_{q}:=$ \emph{exp}$(2\pi\sqrt{-1}/q)$\emph{)} and
acts on $\mathbb{C}^{2}=$ \emph{Spec}$(\mathbb{C}[z_{1},z_{2}])$ linearly
and effectively.
\end{proposition}

\begin{proof}
\noindent See \cite[Proposition 10.1.2, pp. 460-461]{CLS}, \cite[\S\ 2.2,
pp. 32-34]{Fulton} and \cite[Proposition 1.24, p.30]{Oda}.
\end{proof}

\noindent{}By Proposition \ref{ISO} these two numbers $p=p_{\sigma}$ and $\
q=q_{\sigma}$ parametrize uniquely the isomorphism class of the germ\emph{\ }%
$(U_{\sigma}$\emph{, }orb$\left( \sigma\right) )$, up to replacement of\emph{%
\ }$p$ by its socius $\widehat{p}$ (which corresponds just to the
interchange of the coordinates). [The \textit{socius} $\widehat{p}$ of $p$
is defined to be the uniquely determined integer, so that $0\leq \widehat{p}%
<q\emph{,}$ gcd$(\widehat{p},q)=1,$ and $p\,\widehat{p}\equiv1$(mod $q$).]

\begin{proposition}
\label{ISO}Let $\sigma,\tau\subset\mathbb{R}^{2}$ be two $2$-dimensional,
rational, stronly convex polyhedral cones. Then the following conditions are
equivalent\emph{:} \smallskip\newline
\emph{(i) \ }There is a $\mathbb{T}$-equivariant isomorphism $%
U_{\sigma}\cong U_{\tau}$ mapping \emph{orb}$\left( \sigma\right) $ onto 
\emph{orb}$\left( \tau\right) $.\smallskip\newline
\emph{(ii)} There is a unimodular transformation $\Psi:\mathbb{R}%
^{2}\rightarrow\mathbb{R}^{2},$ $\Psi\left( \mathbf{x}\right) :=\Xi\,\mathbf{%
x,}$ $\Xi\in$ \emph{GL}$_{2}(\mathbb{Z}),$ such that\emph{\ }$\Psi\left(
\sigma\right) =\tau.\smallskip$\newline
\emph{(iii)} For the numbers $p_{\sigma},$ $p_{\tau},$ $q_{\sigma},$ $%
q_{\tau}$ associated to $\sigma,\tau$ \emph{(}by \emph{Lemma} \emph{\ref%
{pequ})} we have $q_{\tau }=q_{\sigma}$ and either $p_{\tau}=p_{\sigma}$ or $%
p_{\tau}=\widehat {p}_{\sigma}.$
\end{proposition}

\begin{proof}
See \cite[Proposition 2.4]{Dais2}.
\end{proof}

\section{Compact toric surfaces\label{COMTORS}}

\noindent {}Every compact toric surface is a $2$-dimensional toric variety $%
X_{\Delta }$ associated to a \textit{complete} fan $\Delta $ in $\mathbb{R}%
^{2},$ i.e., a fan having $2$-dimensional cones as maximal cones and whose
support $\left\vert \Delta \right\vert $ is the entire $\mathbb{R}^{2}$ (see 
\cite[Theorem 1.11, p. 16]{Oda}). Consider a complete fan $\Delta $ in $%
\mathbb{R}^{2}$ and suppose that 
\begin{equation}
\sigma _{i}=\mathbb{R}_{\geq 0}\mathbf{n}_{i}+\mathbb{R}_{\geq 0}\mathbf{n}%
_{i+1},\ \ \ i\in \{1,\ldots ,\nu \},  \label{MANYCONES}
\end{equation}%
are its $2$-dimensional cones (with $\nu \geq 3$ and $\mathbf{n}_{i}\in 
\mathbb{Z}^{2}$ primitive for all $i\in \{1,\ldots ,\nu \}$), enumerated in
such a way that $\mathbf{n}_{1},\ldots ,\mathbf{n}_{\nu }$ go \textit{%
anticlockwise} around the origin exactly once in this order (under the usual
convention: $\mathbf{n}_{\nu +1}:=\mathbf{n}_{1},$ $\mathbf{n}_{0}:=\mathbf{n%
}_{\nu }$). $X_{\Delta }$ is obtained by gluing the affine charts $U_{\sigma
_{i}}$ along the open subsets which are defined by the rays $\sigma _{i}\cap
\sigma _{i+1},$ for all $i\in \{1,\ldots ,\nu \}$ (cf. \cite[Theorem 1.4, p.
7]{Oda}). Since $\Delta $ is simplicial, the Picard number $\rho (X_{\Delta
})$ of $X_{\Delta }$ (i.e., the rank of its Picard group Pic$(X_{\Delta })$)
equals%
\begin{equation}
\rho (X_{\Delta })=\text{ }\nu -2,  \label{Picardnr}
\end{equation}%
(see \cite[p. 65]{Fulton}). Now suppose that $\sigma _{i}$ is a $%
(p_{i},q_{i})$-cone for all $i\in \{1,\ldots ,\nu \}$ and introduce the
notation%
\begin{equation}
I_{\Delta }:=\left\{ \left. i\in \{1,\ldots ,\nu \}\ \right\vert \
q_{i}>1\right\} ,\ \ J_{\Delta }:=\left\{ \left. i\in \{1,\ldots ,\nu \}\
\right\vert \ q_{i}=1\right\} ,  \label{IJNOT}
\end{equation}%
to separate the indices corresponding to non-basic from those corresponding
to basic cones. By Propositions \ref{KEY-PROP} and \ref{Sing-Prop} the
singular locus of $X_{\Delta }$ equals%
\begin{equation*}
\text{Sing}(X_{\Delta })=\left\{ \left. \text{orb}(\sigma _{i})\ \right\vert
\ i\in I_{\Delta }\right\} .
\end{equation*}%
For all $i\in I_{\Delta }$ consider the negative-regular continued fraction
expansion of 
\begin{equation*}
\frac{q_{i}}{q_{i}-p_{i}}=b_{1}^{(i)}-\frac{1}{b_{2}^{(i)}-\dfrac{%
\begin{array}{c}
1%
\end{array}%
}{%
\begin{array}{cc}
\ddots  &  \\ 
& b_{s_{i}-1}^{(i)}-\dfrac{1}{b_{s_{i}}^{(i)}}%
\end{array}%
}}\ \ ,
\end{equation*}%
and\ define $\mathbf{u}_{1}^{(i)}:=\mathbf{n}_{i},\ \mathbf{u}_{1}^{(i)}:=%
\frac{1}{q_{i}}((q_{i}-p_{i})\mathbf{n}_{i}+\mathbf{n}_{i+1}),$ and 
\begin{equation*}
\mathbf{u}_{j+1}^{(i)}=b_{j}^{(i)}\mathbf{u}_{j}^{(i)}-\mathbf{u}%
_{j-1}^{(i)},\ \ \forall j\in \{1,\ldots ,s_{i}\}.
\end{equation*}%
It is easy to see that $\mathbf{u}_{s_{i}+1}^{(i)}=\mathbf{n}_{i+1},$ and
that $b_{j}^{(i)}$ are integers $\geq 2,$ for all indices $j\in \{1,\ldots
,s_{i}\}.$ According to \cite[Proposition 4.9, p. 99]{Dais1}, the
self-intersection number of the canonical divisor $K_{X_{\Delta }}$ of $%
X_{\Delta }$ equals 
\begin{equation}
K_{X_{\Delta }}^{2}=12-\nu +\sum_{i\in I_{\Delta }}\left( \tfrac{%
q_{i}-p_{i}+1}{q_{i}}+\tfrac{q_{i}-\widehat{p}_{i}+1}{q_{i}}%
-2+\sum_{j=1}^{s_{i}}\left( b_{j}^{(i)}-3\right) \right) .
\label{SELFINTFORMULA}
\end{equation}%
By construction, the birational morphism $f:X_{\widetilde{\Delta }%
}\longrightarrow X_{\Delta }$ induced by the refinement%
\begin{equation*}
\widetilde{\Delta }:=\left\{ 
\begin{array}{c}
\text{ the cones }\left\{ \left. \sigma _{i}\ \right\vert \ i\in J_{\Delta
}\right\} \text{ and } \\ 
\left\{ \left. \mathbb{R}_{\geq 0}\,\mathbf{u}_{j}^{(i)}+\mathbb{R}_{\geq
0}\,\mathbf{u}_{j+1}^{(i)}\ \right\vert \ i\in I_{\Delta },\ j\in
\{0,1,\ldots ,s_{i}\}\right\} , \\ 
\text{together with their faces}%
\end{array}%
\right\} .
\end{equation*}%
of the fan $\Delta $ is the\textit{\ minimal desingularization} of $%
X_{\Delta }.$ The \textit{exceptional divisor} 
\begin{equation*}
E^{(i)}:=\sum_{j=1}^{s_{i}}E_{j}^{(i)},\ i\in I_{\Delta },
\end{equation*}%
replacing orb$(\sigma _{i})$ via $f$ has 
\begin{equation*}
E_{j}^{(i)}:=\text{ }\overline{\text{orb}_{\widetilde{\Delta }}(\mathbb{R}%
_{\geq 0}\,\mathbf{u}_{j}^{(i)})}\ (\cong \mathbb{P}_{\mathbb{C}}^{1}),\text{
\ }\forall j\in \{1,2,\ldots ,s_{i}\},
\end{equation*}%
(i.e., the closures of the orbits of the new\ rays w.r.t. $\widetilde{\Delta 
}$) as its components, and self-intersection number $%
(E_{j}^{(i)})^{2}=-b_{j}^{(i)}.$ Moreover, $\overline{C}_{i}:=\overline{%
\text{orb}_{\widetilde{\Delta }}(\mathbb{R}_{\geq 0}\,\mathbf{n}_{i})}$ is
the \textit{strict transform }of $C_{i}:=\overline{\text{orb}_{\Delta }(%
\mathbb{R}_{\geq 0}\,\mathbf{n}_{i})}$ w.r.t. $f$ for all $i\in \{1,2,\ldots
,\nu \}.$

\begin{definition}
\label{INTROOFRis}For every $i\in\{1,\ldots,\nu\}$ we introduce integers $%
r_{i}$ \textit{uniquely determined} by the conditions:%
\begin{equation}
r_{i}\mathbf{n}_{i}=\left\{ 
\begin{array}{ll}
\mathbf{u}_{s_{i-1}}^{(i-1)}+\mathbf{u}_{1}^{(i)}, & \text{if }i\in
I_{\Delta }^{\prime}, \\ 
\mathbf{n}_{i-1}+\mathbf{u}_{1}^{(i)}, & \text{if\emph{\ }}i\in I_{\Delta
}^{\prime\prime}, \\ 
\mathbf{u}_{s_{i-1}}^{(i-1)}+\mathbf{n}_{i+1}, & \text{if }i\in J_{\Delta
}^{\prime}, \\ 
\mathbf{n}_{i-1}+\mathbf{n}_{i+1}, & \text{if }i\in
J_{\Delta}^{\prime\prime},%
\end{array}
\right.  \label{CONDri}
\end{equation}
where%
\begin{equation*}
I_{\Delta}^{\prime}:=\left\{ \left. i\in I_{\Delta}\ \right\vert \
q_{i-1}>1\right\} ,\ \ I_{\Delta}^{\prime\prime}:=\left\{ \left. i\in
I_{\Delta}\ \right\vert \ q_{i-1}=1\right\} ,
\end{equation*}
and%
\begin{equation*}
J_{\Delta}^{\prime}:=\left\{ \left. i\in J_{\Delta}\ \right\vert \
q_{i-1}>1\right\} ,\ \ J_{\Delta}^{\prime\prime}:=\left\{ \left. i\in
J_{\Delta}\ \right\vert \ q_{i-1}=1\right\} ,
\end{equation*}
with $I_{\Delta},J_{\Delta}$ as in (\ref{IJNOT}).
\end{definition}

\noindent{}By \cite[Lemma 4.3]{Dais1}, for $i\in\{1,\ldots,\nu\},$ $-r_{i}$
is nothing but the self-intersection number $\overline{C}_{i}^{2}$ of 
\textit{\ }$\overline{C}_{i}.$ The triples $(p_{i},q_{i},r_{i}),$ $%
i\in\{1,2,\ldots,\nu\},$ are used to define the \textsc{wve}$^{2}$\textsc{c}%
-graph\textit{\ }$\mathfrak{G}_{\Delta}.$

\begin{definition}
A \textit{circular graph }is a plane graph whose vertices are points on a
circle and whose edges are the corresponding arcs (on this circle, each of
which connects two consecutive vertices). We say that a circular graph $%
\mathfrak{G}$ is $\mathbb{Z}$-\textit{weighted at its vertices} and \textit{%
double} $\mathbb{Z}$-\textit{weighted} \textit{at its edges} (and call it 
\textsc{wve}$^{2}$\textsc{c}-\textit{graph}, for short) if it is accompanied
by two maps 
\begin{equation*}
\left\{ \text{Vertices of }\mathfrak{G}\right\} \longmapsto\mathbb{Z},\
\left\{ \text{Edges of }\mathfrak{G}\right\} \longmapsto\mathbb{Z}^{2},
\end{equation*}
assigning to each vertex an integer and to each edge a pair of integers,
respectively. To every complete fan $\Delta$ in $\mathbb{R}^{2}$ (as
described above) we associate an anticlockwise directed \textsc{wve}$^{2}$%
\textsc{c}-graph $\mathfrak{G}_{\Delta}$ with 
\begin{equation*}
\left\{ \text{Vertices of }\mathfrak{G}_{\Delta}\right\} =\{\mathbf{v}%
_{1},\ldots,\mathbf{v}_{\nu}\}\text{\ }\ \text{and\ \ }\left\{ \text{Edges
of }\mathfrak{G}_{\Delta}\right\} =\{\overline{\mathbf{v}_{1}\mathbf{v}_{2}}%
,\ldots,\overline{\mathbf{v}_{\nu}\mathbf{v}_{1}}\},
\end{equation*}
$(\mathbf{v}_{\nu+1}:=\mathbf{v}_{1}),$ by defining its \textquotedblleft
weights\textquotedblright\ as follows: 
\begin{equation*}
\mathbf{v}_{i}\longmapsto-r_{i},\text{ \ }\ \ \overline{\mathbf{v}_{i}%
\mathbf{v}_{i+1}}\longmapsto\left( p_{i},q_{i}\right) ,\ \forall
i\in\{1,\ldots,\nu\}.
\end{equation*}
The \textit{reverse graph} $\mathfrak{G}_{\Delta}^{\text{rev}}$ of $%
\mathfrak{G}_{\Delta}$ is the directed \textsc{wve}$^{2}$\textsc{c}-graph
which is obtained by changing the double weight $\left( p_{i},q_{i}\right) $
of the edge $\overline{\mathbf{v}_{i}\mathbf{v}_{i+1}}$ into $(\widehat{p}%
_{i},q_{i})$ and reversing the initial anticlockwise direction of $\mathfrak{%
G}_{\Delta}$ into clockwise direction (see Figure \ref{Fig.3}).
\end{definition}

\begin{figure}[ht]
\includegraphics[height=5cm, width=10cm]{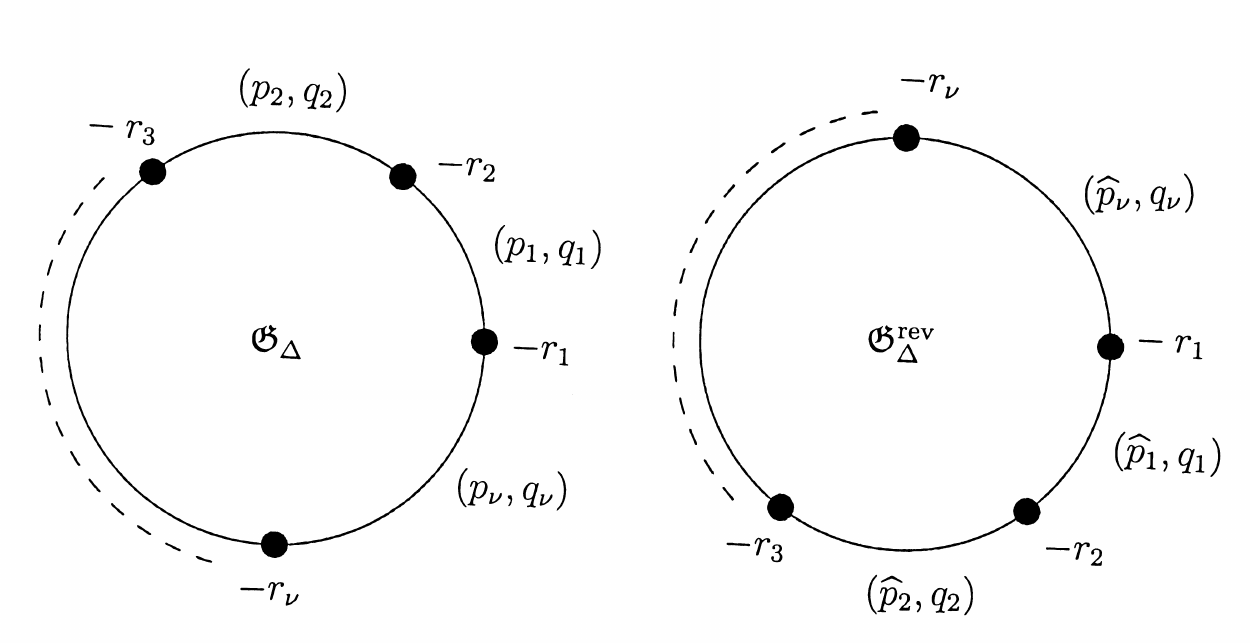}
\caption{}
\label{Fig.3}
\end{figure}

\begin{theorem}
\label{CLASSIFTHM}Let $\Delta,$ $\Delta^{\prime}$ be two complete fans in $%
\mathbb{R}^{2}.$ Then the following conditions are equivalent\emph{%
:\smallskip\ }\newline
\emph{(i)} The compact toric surfaces $X_{\Delta}$ and $X_{\Delta^{\prime}}$
are isomorphic.\smallskip \ \newline
\emph{(ii)} Either $\mathfrak{G}_{\Delta^{\prime}}\overset{\text{\emph{gr.}}}%
{\cong}\mathfrak{G}_{\Delta}$ or $\mathfrak{G}_{\Delta^{\prime}}\overset{%
\text{\emph{gr.}}}{\cong}\mathfrak{G}_{\Delta }^{\text{\emph{rev}}}.$
\end{theorem}

\noindent Here \textquotedblleft$\,\overset{\text{gr.}}{\cong}\,$%
\textquotedblright\ indicates graph-theoretic isomorphism (i.e., a bijection
between the sets of vertices which preserves the corresponding weights). For
further details and for the proof of Theorem \ref{CLASSIFTHM} (which can be
viewed as an appropriate generalization of Proposition \ref{ISO} for
complete fans in $\mathbb{R}^{2}$) the reader is referred to \cite[\S 5]%
{Dais1}. [Convention: To be absolutely compatible with Oda's circular graphs
we omit the weights of the edges which are equal to $(0,1)$, i.e., those
corresponding to basic cones, whenever we draw a \textsc{wve}$^{2}$\textsc{c}%
-graph.]

\section{Toric log del Pezzo surfaces and LDP-polygons\label{LDPPOL}}

\begin{definition}
\noindent{}Let $Q\subset\mathbb{R}^{2}$ be a convex polygon. Denote by $%
\mathcal{V}(Q)$ and $\mathcal{F}(Q)$ the set of its vertices and the set of
its facets (edges), respectively. $Q$ is called an \textit{LDP-polygon} if
it contains the origin in its interior, and its vertices belong to $\mathbb{Z%
}^{2}$ and are primitive. (Obviously, the image of an LDP-polygon under a
unimodular transformation is again an LDP-polygon.)
\end{definition}

\noindent{}If $Q$ is an LDP-polygon, we shall denote by $X_{Q}$ the compact
toric surface $X_{\Delta_{Q}}$ constructed by means of the fan 
\begin{equation*}
\Delta_{Q}:=\left\{ \left. \text{the cones }\sigma_{F}\,\text{\ together
with their faces\ }\right\vert \,F\in\mathcal{F}(Q)\right\} ,
\end{equation*}
where $\sigma_{F}:=\left\{ \left. \lambda\mathbf{x}\,\right\vert \,\mathbf{x}%
\in F\text{ and }\lambda\in\mathbb{R}_{\geq0}\right\} $ for all $F\in%
\mathcal{F}(Q).$

\begin{proposition}
\label{ISOLDP}\emph{(i)} A compact toric surface is log del Pezzo if and
only if it is isomorphic to $X_{Q}$ for some LDP-polygon $Q.\smallskip$%
\newline
\emph{(ii)} There is a one-to-one correspondence 
\begin{equation*}
\left\{ 
\begin{array}{c}
\text{\emph{lattice-equivalence }} \\ 
\text{\emph{classes}} \\ 
\text{\emph{of LDP-polytopes}}%
\end{array}
\right\} \ni\left[ Q\right] \longmapsto\left[ X_{Q}\right] \in\left\{ 
\begin{array}{c}
\text{\emph{isomorphism classes }} \\ 
\text{\emph{of toric log del Pezzo}} \\ 
\text{\emph{surfaces}}%
\end{array}
\right\} .
\end{equation*}
\end{proposition}

\begin{proof}
(i) This follows from \cite[Remark 6.7, p. 107]{Dais1}.\smallskip \ \newline
(ii) If $Q$ is an LDP-polygon, $\Psi:\mathbb{R}^{2}\rightarrow \mathbb{R}%
^{2},$ $\Psi\left( \mathbf{x}\right) :=\Xi\,\mathbf{x,}$ $\Xi\in$ GL$_{2}(%
\mathbb{Z}),$ a unimodular transformation, and $Q^{\prime}:=\Psi(Q),$ then%
\begin{equation*}
\mathfrak{G}_{\Delta_{Q^{\prime}}}\overset{\text{gr.}}{\cong}\mathfrak{G}%
_{\Delta_{Q}},\text{ whenever }\det(\Xi)=1,\text{ and }\mathfrak{G}%
_{\Delta_{Q^{\prime}}}\overset{\text{gr.}}{\cong}\mathfrak{G}_{\Delta_{Q}}^{%
\text{rev}},\text{ whenever }\det(\Xi)=-1.
\end{equation*}
By Theorem \ref{CLASSIFTHM}, $X_{Q}$ and $X_{Q^{\prime}}$ are isomorphic.
And conversely, if $X_{Q}$ and $X_{Q^{\prime}}$ are isomorphic for some
LDP-polygons $Q,Q^{\prime},$ then 
\begin{equation}
\text{either \ }\mathfrak{G}_{\Delta_{Q^{\prime}}}\overset{\text{gr.}}{\cong}%
\mathfrak{G}_{\Delta_{Q}}\ \ \text{or\ \ }\mathfrak{G}_{\Delta_{Q^{\prime}}}%
\overset{\text{gr.}}{\cong}\mathfrak{G}_{\Delta_{Q}}^{\text{rev}}.
\label{TWOISO}
\end{equation}
Thus, by (\ref{TWOISO}) there exists an automorphism $\varpi$ of the lattice 
$\mathbb{Z}^{2}=\mathbb{Z}\tbinom{1}{0}\oplus\mathbb{Z}\tbinom{0}{1} $ with 
\begin{equation*}
\det(\varpi)=\left\{ 
\begin{array}{rr}
1, & \text{in the first case,} \\ 
-1, & \text{in the second case,}%
\end{array}
\right.
\end{equation*}
such that $\varpi_{\mathbb{R}}(\Delta_{Q})=\Delta_{Q^{\prime}}$
(preserving/reversing the ordering of the cones), where%
\begin{equation*}
\varpi_{\mathbb{R}}:=\varpi\otimes_{\mathbb{Z}}\text{id}_{\mathbb{R}}:%
\mathbb{R}^{2}\longrightarrow\mathbb{R}^{2}
\end{equation*}
denotes its scalar extension. Obviously, $\varpi_{\mathbb{R}}(Q)=Q^{\prime}.$
\end{proof}

\begin{note}
Let $Q$ be an arbitrary LDP-polygon. For each $F\in \mathcal{F}(Q)$ assume
that $\sigma _{F}$ is a $(p_{F},q_{F})$-cone$.$ Then from \cite[Lemma 6.8]%
{Dais1} one concludes that the index $\ell $ of $X_{Q}$ equals%
\begin{equation}
\ell =\func{lcm}\left\{ \left. l_{F}\right\vert F\in \mathcal{F}(Q)\right\} 
\text{ with }l_{F}:=\frac{q_{F}}{\gcd (q_{F},p_{F}-1)}.  \label{INDEXFORMEL}
\end{equation}%
{} If we consider the \textit{polar polygon} $\mathring{Q}:=\left\{ \left. 
\mathbf{y}\in \text{Hom}_{\mathbb{R}}(\mathbb{R}^{2},\mathbb{R})\right\vert
\left\langle \mathbf{y},\mathbf{x}\right\rangle \geq -1,\ \forall \mathbf{x}%
\in Q\right\} $ of $Q,$ where $\left\langle \cdot ,\cdot \right\rangle :$ Hom%
$_{\mathbb{R}}(\mathbb{R}^{2},\mathbb{R})\times \mathbb{R}^{2}\rightarrow 
\mathbb{R}$ denotes the usual inner product, then $\mathring{Q}$ contains
the origin in its interior, and the index $\ell $ of $X_{Q}$ equals 
\begin{equation*}
\ell =\min \left\{ \left. \kappa \in \mathbb{Z}_{>0}\right\vert \mathcal{V}%
(\kappa \mathring{Q})\subset \mathbb{Z}^{2}\right\} \ \ \ (\text{ with }%
\kappa \mathring{Q}:=\left\{ \kappa \mathbf{y}\left\vert \mathbf{y}\in 
\mathring{Q}\right. \right\} ).
\end{equation*}%
Moreover, if $F\in \mathcal{F}(Q),$ denoting by $\boldsymbol{\eta }_{F}$ the
unique primitive $\boldsymbol{\eta }_{F}\in \mathbb{Z}^{2}$ for which $%
\left\langle \boldsymbol{\eta }_{F},\mathbf{x}\right\rangle =l_{F},\ \forall 
\mathbf{x}\in F,$ we have 
\begin{equation}
\mathcal{V}(\mathring{Q})=\left\{ \left. \frac{-1}{l_{F}}\boldsymbol{\eta }%
_{F}\right\vert F\in \mathcal{F}(Q)\right\} .  \label{VERTICESPOLAR}
\end{equation}
\end{note}

\section{Proof of the classification theorem \protect\ref{MAIN1}\label%
{PROOFOFMAIN1}}

\noindent {}Let $Q$ be an LDP-polygon with vertex set $\mathcal{V}%
(Q)=\left\{ \mathbf{n}_{1},\ldots ,\mathbf{n}_{\nu }\right\} ,$ $\nu \geq 3.$
Assume that $\sigma _{i},$ $i\in \{1,\ldots ,\nu \},$ are the $2$%
-dimensional cones of $\Delta _{Q},$ defined and ordered (anticlockwise) as
in (\ref{MANYCONES}), and that only one of these cones, say $\sigma _{1},$
is a non-basic $(p,q)$-cone (i.e., $q>1$). By Lemma \ref{pequ}, there is a
unimodular transformation $\Psi _{1}:\mathbb{R}^{2}\longrightarrow \mathbb{R}%
^{2},$ $\Psi _{1}\left( \mathbf{x}\right) :=\Xi \,\mathbf{x,}$ $\Xi \in $ GL$%
_{2}(\mathbb{Z}),$ such that%
\begin{equation*}
\Psi _{1}\left( \sigma _{1}\right) =\mathbb{R}_{\geq 0}\tbinom{1}{0}+\mathbb{%
R}_{\geq 0}\tbinom{p}{q}.
\end{equation*}%
Without loss of generality we may assume that $\det (\Xi )=1$ (because
otherwise the proof of Theorem \ref{MAIN1} which follows can be performed
similarly if one works with the vertices ordered clockwise). This means that 
$\Psi _{1}(\mathbf{n}_{1})=\tbinom{1}{0}$ and $\Psi _{1}(\mathbf{n}_{2})=%
\tbinom{p}{q}.$ We set $\mathbf{w}_{i}:=\Psi _{1}(\mathbf{n}_{i}),$ for all $%
i\in \{1,\ldots ,\nu \}$ (and $\mathbf{w}_{\nu +1}:=\mathbf{w}_{1}$). Since
all cones of $\Delta _{\Psi _{1}\left( Q\right) }$ are strongly convex and $%
\left\vert \Delta _{\Psi _{1}\left( Q\right) }\right\vert =\mathbb{R}^{2}$, 
\begin{equation}
\exists \mu \in \{3,\ldots ,\nu \}:\mathbf{w}_{\mu }=\tbinom{a}{b}\in
\left\{ \left. \tbinom{x}{y}\in \mathbb{R}^{2}\right\vert x<0\right\} \cap 
\mathbb{Z}^{2}.  \label{DEFVM}
\end{equation}

\begin{lemma}
\label{BOTHBASIC}\emph{(i)} The cones $\mathbb{R}_{\geq 0}\mathbf{w}_{\mu }+%
\mathbb{R}_{\geq 0}\mathbf{w}_{1}$ and $\mathbb{R}_{\geq 0}\mathbf{w}_{2}+%
\mathbb{R}_{\geq 0}\mathbf{w}_{\mu }$ are basic.\smallskip\ \newline
\emph{(ii)} $q=p+1$ \emph{(}and consequently, $\widehat{p}=p$ and $\mathbf{w}%
_{\mu }=\tbinom{-1}{-1}$\emph{).}
\end{lemma}

\begin{proof}
(i) Using Proposition \ref{KEY-PROP} it suffices to prove that {\small 
\begin{equation}
\text{conv}(\{\mathbf{0},\mathbf{w}_{\mu },\mathbf{w}_{1}\})\cap \mathbb{Z}%
^{2}=\{\mathbf{0},\mathbf{w}_{\mu },\mathbf{w}_{1}\},\text{ conv}(\{\mathbf{0%
},\mathbf{w}_{2},\mathbf{w}_{\mu }\})\cap \mathbb{Z}^{2}=\{\mathbf{0},%
\mathbf{w}_{2},\mathbf{w}_{\mu }\}.  \label{CONVCONV}
\end{equation}%
} \noindent Obviously, $\mathcal{V}(\Psi _{1}\left( Q\right) )\mathbb{r}\{%
\mathbf{w}_{\mu },\mathbf{w}_{1},\mathbf{w}_{2}\}$ is either empty or a
subset of $(\mathcal{U}_{1}\cup \mathcal{U}_{2})\cap \mathbb{Z}^{2},$ where%
\begin{equation*}
\mathcal{U}_{1}:=\left\{ \tbinom{x}{y}\in \mathbb{R}^{2}\left\vert y<0,\ \
y<x,\text{ and }qx-(p-1)y<q\right. \right\} ,
\end{equation*}%
and%
\begin{equation*}
\mathcal{U}_{2}:=\left\{ \tbinom{x}{y}\in \mathbb{R}^{2}\left\vert qx<py,\ \
y>x,\text{ and }qx-(p-1)y<q\right. \right\} .
\end{equation*}%
($\left\{ \left. \tbinom{x}{y}\in \mathbb{R}^{2}\right\vert
qx-(p-1)y=q\right\} $ is the supporting line of the edge conv$(\{\mathbf{w}%
_{1},\mathbf{w}_{2}\})$ of $\Psi _{1}\left( Q\right) .$) If conv$(\{\mathbf{w%
}_{\mu },\mathbf{w}_{1}\})\in \mathcal{F}(\Psi _{1}\left( Q\right) ),$ i.e.,
if $\mu =\nu ,$ the first equality in (\ref{CONVCONV}) is obvious (because $%
\Psi _{1}\left( \sigma _{\nu }\right) $ is basic by definition). If conv$(\{%
\mathbf{w}_{\mu },\mathbf{w}_{1}\})\notin \mathcal{F}(\Psi _{1}\left(
Q\right) ),$ then $\mathcal{V}(\Psi _{1}\left( Q\right) )\cap \mathcal{U}%
_{1}\neq \varnothing ,$ and if we would assume that 
\begin{equation*}
\exists \mathbf{m}\in (\text{conv}(\{\mathbf{0},\mathbf{w}_{\mu },\mathbf{w}%
_{1}\})\cap \mathbb{Z}^{2})\mathbb{r}\{\mathbf{0},\mathbf{w}_{\mu },\mathbf{w%
}_{1}\},
\end{equation*}%
then there would be a%
\begin{equation*}
\xi \in \{\mu +1,\mu +2,\ldots ,\nu ,\nu +1\}:\mathbf{m}\in (\text{conv}(\{%
\mathbf{0},\mathbf{w}_{\xi -1},\mathbf{w}_{\xi }\})\cap \mathbb{Z}^{2})%
\mathbb{r}\{\mathbf{0},\mathbf{w}_{\xi -1},\mathbf{w}_{\xi }\},
\end{equation*}%
leading to contradiction (because $\Psi _{1}\left( \sigma _{\xi -1}\right) $
is basic by definition). Similar arguments (using $\mathcal{U}_{2}$ instead
of $\mathcal{U}_{1}$) show that the second equality in (\ref{CONVCONV}) is
also true.\smallskip\ \newline
(ii) By (i), $\left\vert \det (\mathbf{w}_{\mu },\mathbf{w}_{1})\right\vert
=\left\vert \det (\mathbf{w}_{2},\mathbf{w}_{\mu })\right\vert =1,$ i.e., $%
b\in \{\pm 1\},$ and one of the following conditions is satisfied:{\small 
\begin{eqnarray}
&&%
\begin{array}{ccc}
b=1 & \text{and} & aq-p=1,%
\end{array}
\label{COND1} \\
&&%
\begin{array}{ccc}
b=1 & \text{and} & aq-p=-1,%
\end{array}
\label{COND2} \\
&&%
\begin{array}{ccc}
b=-1 & \text{and} & aq+p=1,%
\end{array}
\label{COND3} \\
&&%
\begin{array}{ccc}
b=-1 & \text{and} & aq+p=-1.%
\end{array}
\label{COND4}
\end{eqnarray}%
} (\ref{COND1}) gives $a=\tfrac{1+p}{q}>0$ which is not true (because $a<0$
by (\ref{DEFVM})). Similarly, (\ref{COND2}) is not true (as we would have $a=%
\tfrac{p-1}{q}\geq 0$ ). In case (\ref{COND3}), 
\begin{equation*}
a=\tfrac{-(p-1)}{q}\Rightarrow q\mid p-1\Rightarrow p<q\leq p-1\text{ (being
absurd).}
\end{equation*}%
Therefore, (\ref{COND4}) is necessarily true and $a=-\tfrac{p+1}{q}$. Now
since $q\mid p+1$ and $p<q,$ we have%
\begin{equation*}
q=p+1\Rightarrow \lbrack a=-1,\text{ }b=-1]\Rightarrow \mathbf{w}_{\mu }=%
\tbinom{-1}{-1}
\end{equation*}%
and $p+1\mid \left( p^{2}-1\right) \Rightarrow \widehat{p}=p.$
\end{proof}

\begin{lemma}
There exists a unimodular transformation $\Psi _{2}:\mathbb{R}%
^{2}\rightarrow \mathbb{R}^{2}$ such that 
\begin{equation*}
\Psi _{2}(\Psi _{1}\left( \sigma _{1}\right) )=\mathbb{R}_{\geq 0}\tbinom{1}{%
-1}+\mathbb{R}_{\geq 0}\tbinom{p}{1},
\end{equation*}%
with $\Psi _{2}\tbinom{1}{0}=\tbinom{1}{-1},$ $\Psi _{2}\tbinom{p}{p+1}=%
\tbinom{p}{1},$ and $\Psi _{2}\tbinom{-1}{-1}=\tbinom{-1}{0}.$
\end{lemma}

\begin{proof}
It is enough to define $\Psi_{2}(\mathbf{x}):=\left( 
\begin{array}{rc}
\QATOPD..{1}{-1} & \QATOPD..{0}{1}%
\end{array}
\right) \mathbf{x},\ \forall\mathbf{x}\in\mathbb{R}^{2}$.
\end{proof}

\noindent {}Next, we set $\Upsilon :=\Psi _{2}\circ \Psi _{1}$, $\mathbf{v}%
_{i}:=\Upsilon (\mathbf{n}_{i}),$ for all $i\in \{1,\ldots ,\nu \}$ (and $%
\mathbf{v}_{\nu +1}:=\mathbf{v}_{1}$). Starting with the minimal generators $%
\mathbf{v}_{1}=\tbinom{1}{-1},$ $\mathbf{v}_{2}=\tbinom{p}{1}$ of the unique
non-basic cone $\Upsilon \left( \sigma _{1}\right) $ of $\Delta _{\Upsilon
\left( Q\right) },$ and with $\mathbf{v}_{\mu }=\tbinom{-1}{0}\in \mathcal{V}%
(\Upsilon \left( Q\right) ),$ we shall study the restrictions on the
location of the remaining vertices of $\Upsilon \left( Q\right) $ in detail.

\begin{lemma}
\label{NCOLVERT}There is no convex polygon having three collinear vertices.
\end{lemma}

\begin{proof}
This is due to the fact that the vertices of a convex polygon are its
extreme points. (See, e.g., \cite[p. 30 and p. 45]{Brondsted}.)
\end{proof}

\begin{lemma}
\label{KEY-LEMMA}The LDP-polygon $\Upsilon\left( Q\right) $\emph{\ (}with $%
\mathcal{V}(\Upsilon\left( Q\right) )=\left\{ \mathbf{v}_{1},\ldots,\mathbf{v%
}_{\nu}\right\} $\emph{)} has the following properties\emph{:\smallskip} 
\newline
\emph{(i)} Setting $k:=\nu-2,$ we have necessarily $k\in\{1,2,3\}.$
Moreover, $\Upsilon\left( Q\right) =Q_{p}^{\left[ k\right] }$ for $%
k\in\{1,3\},$ and either $\Upsilon\left( Q\right) =Q_{p}^{\left[ 2\right] }$
or $\Upsilon\left( Q\right) =\check{Q}_{p}^{\left[ 2\right] }$ for $k=2,$
where $Q_{p}^{\left[ 1\right] },Q_{p}^{\left[ 2\right] },Q_{p}^{\left[ 3%
\right] }$ are the polygons defined in \emph{(\ref{DEFQS}),} and%
\begin{equation*}
\check{Q}_{p}^{\left[ 2\right] }:=\text{\emph{conv}}\left( \left\{ \tbinom{1%
}{-1},\tbinom{p}{1},\tbinom{-1}{0},\tbinom{0}{-1}\right\} \right) .
\end{equation*}
\emph{(ii)} $Q_{p}^{\left[ 2\right] }$ and $\check{Q}_{p}^{\left[ 2\right] }$
are lattice-equivalent.
\end{lemma}

\begin{proof}
(i) If $\mathcal{U}_{1}^{\prime }:=\left\{ \tbinom{x}{y}\in \Psi _{2}(%
\mathcal{U}_{1})\left\vert y\leq -2\right. \right\} ,$ we claim that $%
\mathcal{U}_{1}^{\prime }\cap \mathcal{V}(\Upsilon \left( Q\right)
)=\varnothing .$ If $\mathbf{v}_{\mu +1}\in \mathcal{U}_{1}^{\prime }\cap 
\mathcal{V}(\Upsilon \left( Q\right) ),$ then we would have $\left\vert \det
(\mathbf{v}_{\mu },\mathbf{v}_{\mu +1})\right\vert \geq 2,$ contradicting to
the basicness of the cone $\Upsilon \left( \sigma _{\mu }\right) .$ If%
\begin{equation*}
\ \text{ }\mathbf{v}_{\mu +1}\in \left\{ \tbinom{x}{y}\in \mathbb{Z}%
^{2}\left\vert x\leq 0\text{, }y=-1\right. \right\} \ \text{and}\ \ \mathbf{v%
}_{\mu +2}\in \mathcal{U}_{1}^{\prime }\cap \mathcal{V}(\Upsilon \left(
Q\right) ),
\end{equation*}%
then we would again have $\left\vert \det (\mathbf{v}_{\mu +1},\mathbf{v}%
_{\mu +2})\right\vert \geq 2,$ contradicting to the basicness of the cone $%
\Upsilon \left( \sigma _{\mu +1}\right) .$ Repeating successively this
procedure (until we arrive at $\mathbf{v}_{\nu }$) we bear out our assertion$%
,$ as well as the implication 
\begin{equation*}
\mu \leq \nu -1\Rightarrow \left\{ \mathbf{v}_{\xi }\left\vert \mu +1\leq
\xi \leq \nu \right. \right\} \subset \left\{ \tbinom{x}{y}\in \mathbb{Z}%
^{2}\left\vert x\leq 0\text{, }y=-1\right. \right\} .
\end{equation*}%
Correspondingly, for $\mathcal{U}_{2}^{\prime }:=\left\{ \tbinom{x}{y}\in
\Psi _{2}(\mathcal{U}_{2})\left\vert y\geq 2\right. \right\} $ we show that $%
\mathcal{U}_{2}^{\prime }\cap \mathcal{V}(\Upsilon \left( Q\right)
)=\varnothing ,$ and%
\begin{equation*}
\mu \geq 4\Rightarrow \left\{ \mathbf{v}_{\xi }\left\vert 3\leq \xi \leq \mu
-1\right. \right\} \subset \left\{ \tbinom{x}{y}\in \mathbb{Z}^{2}\left\vert
x\leq p-1\text{, }y=1\right. \right\} .
\end{equation*}%
Hence, $\mathcal{V}(\Upsilon \left( Q\right) )\mathbb{r}\{\mathbf{v}_{1},%
\mathbf{v}_{2},\mathbf{v}_{\mu }\}$ is either empty or a subset of 
\begin{equation*}
\left\{ \tbinom{x}{y}\in \mathbb{Z}^{2}\left\vert x\leq 0\text{ and }%
y=-1\right. \right\} \cup \left\{ \tbinom{x}{y}\in \mathbb{Z}^{2}\left\vert
x\leq p-1\text{ and }y=1\right. \right\} .
\end{equation*}%
Taking into account Lemma \ref{NCOLVERT} we conclude that%
\begin{equation*}
\{\mathbf{v}_{1},\mathbf{v}_{2},\mathbf{v}_{\mu }\}\subseteq \mathcal{V}%
(\Upsilon \left( Q\right) )\subseteq \left\{ \mathbf{v}_{1},\mathbf{v}_{2},%
\tbinom{p-1}{1},\mathbf{v}_{\mu },\tbinom{0}{-1}\right\} .
\end{equation*}%
Therefore, $k\in \{1,2,3\}$ and there are only four possibilities:\smallskip 
\newline
$\bullet $ Case (a): If $k=1,$ then $\nu =\mu =3$ and $\Upsilon \left(
Q\right) =\text{conv}(\{\mathbf{v}_{1},\mathbf{v}_{2},\mathbf{v}%
_{3}\})=Q_{p}^{\left[ 1\right] }.$ $\smallskip $\newline
$\bullet $ Case (b): If $k=2,$ then $\nu =4$ and \textit{either} $\Upsilon
\left( Q\right) =Q_{p}^{\left[ 2\right] },$ $\mu =4,\,$\textit{or} $\Upsilon
\left( Q\right) =\check{Q}_{p}^{\left[ 2\right] },$ $\mu =3.\smallskip $ 
\newline
$\bullet $ Case (c): If $k=3,$ then $\nu =5,$ $\mu =4$ and $\Upsilon \left(
Q\right) =Q_{p}^{\left[ 3\right] }.\smallskip $\newline
(ii) $Q_{p}^{\left[ 2\right] }$ is mapped onto $\check{Q}_{p}^{\left[ 2%
\right] }$ under the unimodular transformation%
\begin{equation*}
\mathfrak{Y}:\mathbb{R}^{2}\longrightarrow \mathbb{R}^{2},\ \ \mathfrak{Y}(%
\mathbf{x}):=\left( 
\begin{array}{cc}
1 & 1-p \\ 
0 & -1%
\end{array}%
\right) \mathbf{x},\ \forall \mathbf{x}\in \mathbb{R}^{2},
\end{equation*}%
and $\mathfrak{Y}(\mathbf{v}_{1})=\mathbf{v}_{2},$ $\mathfrak{Y}(\mathbf{v}%
_{2})=\mathbf{v}_{1},\mathfrak{\ Y}\tbinom{-1}{0}=\tbinom{-1}{0},$ $%
\mathfrak{Y}\tbinom{p-1}{1}=\tbinom{0}{-1}.$
\end{proof}

\begin{note}
The set conv$(\{\mathbf{v}_{1},\mathbf{v}_{2}\})\cap\mathbb{Z}^{2}$ is empty
for $p$ even and consists of the single lattice point $\tbinom{\frac{1}{2}%
\smallskip(p+1)}{0}$ for $p$ odd. Thus, the number of the lattice points
belonging to the boundary of $Q_{p}^{\left[ k\right] },$ $k\in\{1,2,3\},$
equals $k+2$ whenever $p$ is even and $k+3$ whenever $p$ is odd. Since area$%
(Q_{p}^{\left[ k\right] })=\tfrac{p+k}{2}+1,$ Pick's formula gives%
\begin{equation*}
\sharp(\text{int}(Q_{p}^{\left[ k\right] })\cap\mathbb{Z}^{2})=\left\{ 
\begin{array}{ll}
\tfrac{p}{2}+1, & \text{if }p\text{ is even,} \\ 
\tfrac{p-1}{2}+1, & \text{if }p\text{ is odd.}%
\end{array}
\right.
\end{equation*}
\end{note}

\noindent{}\textit{Proof of Theorem }\ref{MAIN1}. (i)-(ii) Up to
isomorphism, every toric log del Pezzo surface with exactly one singularity
is of the form $X_{Q}$ with $Q$ as above. By (\ref{Picardnr}), Lemma \ref%
{KEY-LEMMA} and Proposition \ref{ISOLDP} we infer that the Picard number $%
\rho\left( X_{Q}\right) $ of $X_{Q}$ can take only the values $1,2$ and $3,$
and that for $k\in\{1,2,3\},$%
\begin{equation*}
\rho\left( X_{Q}\right) =k\Longleftrightarrow\exists p\in\mathbb{Z}%
_{>0}:X_{Q}\cong X_{Q_{p}^{\left[ k\right] }}.
\end{equation*}
(Note that for $k=2,$ $\mathfrak{Y}$ induces a graph-theoretic isomorphism $%
\mathfrak{G}_{\Delta_{\check{Q}_{p}^{\left[ 2\right] }}}\overset{\text{gr.}}{%
\cong}\mathfrak{G}_{\Delta_{Q_{p}^{\left[ 2\right] }}}^{\text{rev}},$
meaning that $X_{Q_{p}^{\left[ 2\right] }}\cong X_{\check{Q}_{p}^{\left[ 2%
\right] }}.$) The fan $\widetilde{\Delta}_{Q_{p}^{\left[ k\right] }}$ which
is used to construct the minimal desingularization of $X_{Q_{p}^{\left[ k%
\right] }}$ (as explained in \S \ref{COMTORS}) contains just one additional
ray (compared with $\Delta_{Q_{p}^{\left[ k\right] }}$), namely $\mathbb{R}%
_{\geq0}\tbinom {1}{0}.$ The closure of its orbit constitutes the single
exceptional divisor, say $E,$ w.r.t. this desingularization, with $%
E^{2}=-(p+1).$ Setting $\mathbf{u}_{E}:=\tbinom{1}{0}$ we compute the
integers $r_{i},$ $i\in\{1,\ldots,k+2\},\,\ $(defined in (\ref{CONDri})) in
the three different cases:\smallskip\ \newline
$\bullet$ Case (a): If $k=1,$ then $\mathbf{v}_{1}=\tbinom{1}{-1},$ $\mathbf{%
v}_{2}=\tbinom{p}{1},\mathbf{v}_{3}=\tbinom {-1}{0},$ and%
\begin{equation*}
[\mathbf{v}_{3}+\mathbf{u}_{E}=\mathbf{0}, \ \mathbf{v}_{2}+\mathbf{v}%
_{1}=-(p+1)\mathbf{v}_{3}] \Rightarrow r_{1}=r_{2}=0,\text{ }r_{3}=-(p+1).
\end{equation*}
\newline
$\bullet$ Case (b): If $k=2,$ then $\mathbf{v}_{1}=\tbinom{1}{-1},$ $\mathbf{%
v}_{2}=\tbinom{p}{1},$ $\mathbf{v}_{3}=\tbinom{p-1}{1},$ $\mathbf{v}_{4}=%
\tbinom{-1}{0},$ and%
\begin{equation*}
\left. 
\begin{array}{c}
\mathbf{v}_{4}+\mathbf{u}_{E}=\mathbf{0},\mathbf{u}_{E}+\mathbf{v}_{3}=%
\mathbf{v}_{2} \\ 
\mathbf{v}_{2}+\mathbf{v}_{4}=\mathbf{v}_{3},\mathbf{v}_{3}+\mathbf{v}_{1}=-p%
\mathbf{v}_{4}%
\end{array}
\right\} \Rightarrow r_{1}=0,\text{ }r_{2}=r_{3}=1,\text{ }r_{4}=-p.
\end{equation*}
$\bullet$ Case (c): If $k=3,$ then $\mathbf{v}_{1}=\tbinom{1}{-1},$ $\mathbf{%
v}_{2}=\tbinom{p}{1},$ $\mathbf{v}_{3}=\tbinom{p-1}{1},$ $\mathbf{v}_{4}=%
\tbinom{-1}{0},$ $\mathbf{v}_{5}=\tbinom{0}{-1},$ and 
\begin{equation*}
\left. 
\begin{array}{r}
\mathbf{v}_{5}+\mathbf{u}_{E}=\mathbf{v}_{1},\mathbf{u}_{E}+\mathbf{v}_{3}=%
\mathbf{v}_{2}, \\ 
\mathbf{v}_{2}+\mathbf{v}_{4}=\mathbf{v}_{3},\mathbf{v}_{3}+\mathbf{v}%
_{5}=-(p+1)\mathbf{v}_{4}, \\ 
\mathbf{v}_{4}+\mathbf{v}_{1}=\mathbf{v}_{5}%
\end{array}
\right\} \Rightarrow r_{1}=r_{2}=\text{ }r_{3}=r_{5}=1,\text{ }r_{4}=-(p-1).
\end{equation*}
Hence, the \textsc{wve}$^{2}$\textsc{c}-graphs $\mathfrak{G}_{\Delta
_{Q_{p}^{\left[ k\right] }}}$ are indeed those depicted in Figure \ref{Fig.2}%
.\smallskip

\noindent{}(iii) Defining for every positive integer $p$ the complete fan 
\begin{equation*}
\mathfrak{D}_{p}:=\left\{ 
\begin{array}{c}
\text{ the cones }\mathbb{R}_{\geq0}\tbinom{1}{-1}+\mathbb{R}_{\geq0}\tbinom{%
1}{0},\text{ }\mathbb{R}_{\geq0}\tbinom{1}{0}+\mathbb{R}_{\geq 0}\tbinom{p}{1%
},\medskip\text{ } \\ 
\mathbb{R}_{\geq0}\tbinom{p}{1}+\mathbb{R}_{\geq0}\tbinom{-1}{0},\text{ and }%
\mathbb{R}_{\geq0}\tbinom{-1}{0}+\mathbb{R}_{\geq0}\tbinom{1}{-1},\medskip
\\ 
\text{together with their faces}%
\end{array}
\right\} ,
\end{equation*}
we see that $X_{\mathfrak{D}_{p}}\cong\mathbb{F}_{p+1},$ having $\overline {%
\text{orb}_{\mathfrak{D}_{p}}(\mathbb{R}_{\geq0}\tbinom{1}{0})}$ as its $%
\infty$-section. The surfaces $X_{Q_{p}^{\left[ k\right] }}$ are
characterized as follows:\smallskip\ 

\noindent{}$\bullet$ Case (a): If $k=1,$ then $X_{Q_{p}^{\left[ 1\right]
}}\cong$ $\mathbb{P}_{\mathbb{C}}^{2}(1,1,p+1)$ (see \cite[Lemma 6.1]{Dais2}%
), and it is obtained by contracting the $\infty$-section of $X_{\mathfrak{D}%
_{p}}.$ In fact, since $X_{\mathfrak{D}_{p}}=X_{\widetilde{\Delta}_{Q_{p}^{%
\left[ 1\right] }}}$ is the minimal desingularization of $X_{Q_{p}^{\left[ 1%
\right] }},$ the surface $X_{Q_{p}^{\left[ 1\right] }}$ is nothing but the 
\textit{anticanonical model} of $X_{\mathfrak{D}_{p}}$ (in the sense of
Sakai \cite{Sakai}).\smallskip\ \newline
$\bullet$ Case (b): If $k=2,$ the star subdivision of $\mathfrak{D}_{p-1}$
w.r.t the cone $\mathbb{R}_{\geq0}\tbinom{1}{0}+\mathbb{R}_{\geq0}\tbinom{p-1%
}{1}$ induces the equivariant blow-up $X_{\widetilde{\Delta}_{Q_{p}^{\left[ 2%
\right] }}}\longrightarrow$ $X_{\mathfrak{D}_{p-1}}$ with the orbit of this
cone as centre (cf. \cite[Proposition 3.3.15, p. 130]{CLS}, \cite[Corollary
7.5, p. 45]{Oda-Miyake} or \cite[Theorem VI.7.2, pp. 249-250]{Ewald}). Thus,
the surface $X_{Q_{p}^{\left[ 2\right] }}$ is obtained by contracting the
strict transform of the $\infty$-section of $X_{\mathfrak{D}_{p-1}}$ on $X_{%
\widetilde{\Delta}_{Q_{p}^{\left[ 2\right] }}}.$\smallskip\ 

\noindent{}$\bullet$ Case (c): If $k=3,$ we construct the surface $X_{Q_{p}^{%
\left[ 3\right] }}$ from $X_{Q_{p}^{\left[ 2\right] }}$ by using the
equivariant birational morphism induced by the star subdivision of $%
\mathfrak{D}_{p-1}$ w.r.t the cone $\mathbb{R}_{\geq0}\tbinom{-1}{0}+\mathbb{%
R}_{\geq0}\tbinom{1}{-1}$, i.e., by blowing up its orbit (which is a
non-singular $\mathbb{T}$-fixed point of $X_{Q_{p}^{\left[ 2\right] }}$)$.$

Taking into account that we pass from $X_{\mathfrak{D}_{p-1}}$ to $X_{%
\mathfrak{D}_{p}}$ (and vice versa) by an elementary transformation (cf. 
\cite[Remark 6.3, pp. 105-106]{Dais1}), we illustrate in Figure \ref{Fig.4}
how the equivariant birational morphisms connecting all the above mentioned
compact toric surfaces affect their \textsc{wve}$^{2}$\textsc{c}-graphs.

\begin{figure}[ht]
\includegraphics[height=11cm, width=11.5cm]{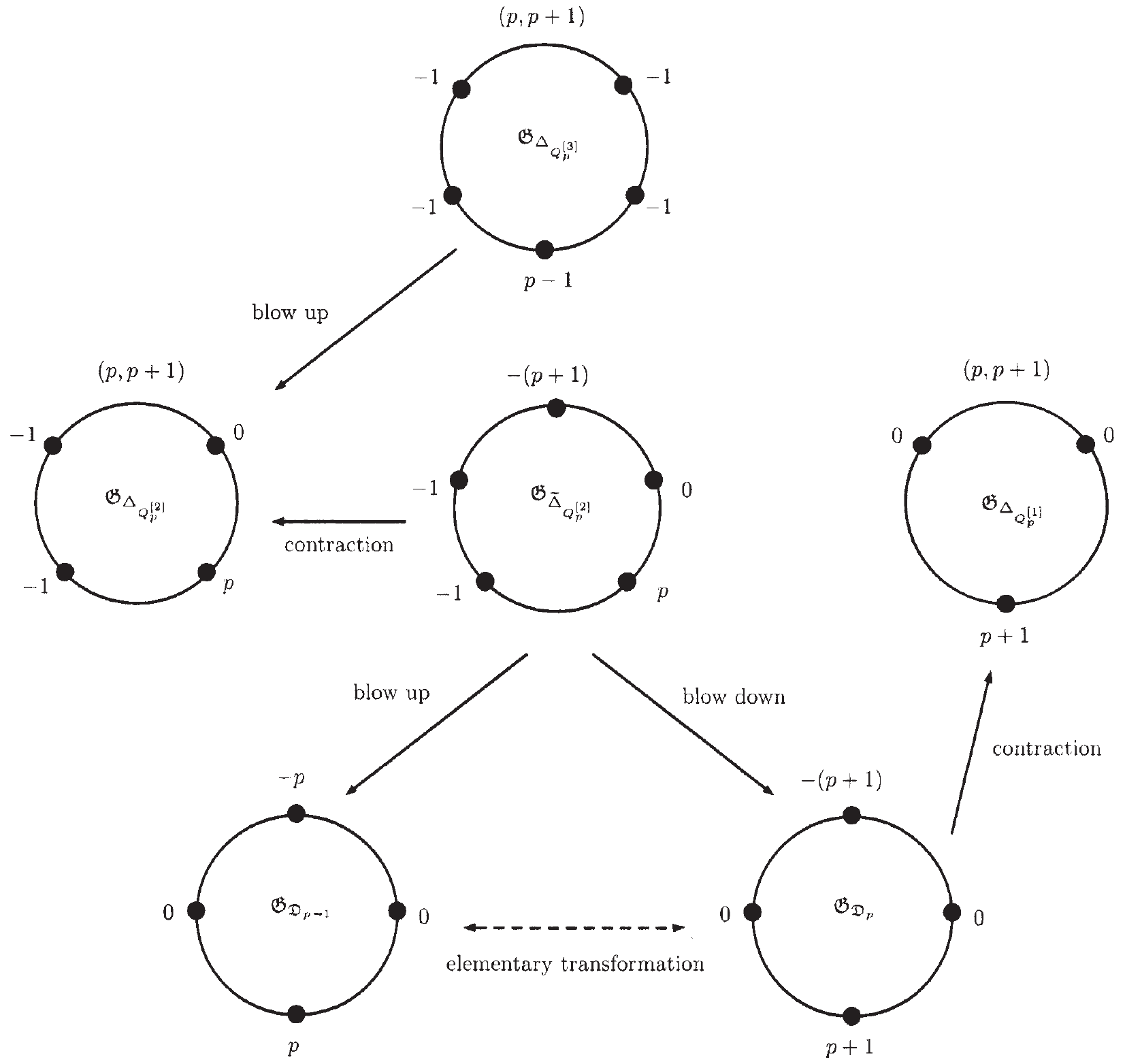}
\caption{}
\label{Fig.4}
\end{figure}

\noindent{}(iv) Since $q=p+1$ and gcd$(p+1,p-1)=$ gcd$(p+1,2)\in\{1,2\},$
formula (\ref{INDEXFORMEL}) shows that the index $\ell$ of $X_{Q}\cong
X_{Q_{p}^{\left[ k\right] }}$ equals $\frac{p+1}{2}$ whenever $p$ is odd and 
$p+1$ whenever $p$ is even. This bears out our assertion about $\ell$.\hfill
{}$\square$

\begin{remark}
Among the LDP-polygons $Q_{p}^{\left[ k\right] },$ only $Q_{1}^{\left[ 1%
\right] },Q_{1}^{\left[ 2\right] },Q_{1}^{\left[ 3\right] }$ are \textit{%
reflexive} (with index $\ell=1$ and a unique \textit{Gorenstein}
singularity).
\end{remark}

\section{Defining equations\label{PROOFOFMAIN2}}

\noindent{}Let $Q$ be an arbitrary LDP-polygon. Since the Cartier divisor $%
-\ell K_{X_{Q}}$ on $X_{Q}$ is very ample, setting 
\begin{equation*}
\delta_{Q}:=\sharp((\ell\mathring{Q})\cap\mathbb{Z}^{2})-1,
\end{equation*}
the complete linear system $\left\vert -\ell K_{X_{Q}}\right\vert $ induces
the closed embedding $\Phi_{\left\vert -\ell K_{X_{Q}}\right\vert }$,

\begin{equation*}
\xymatrix{\mathbb{T} \hspace{0.2cm} \ar@{^{(}->}[rr]^{\iota}
\ar@/_1.7pc/[rrrr] & & X_{Q} \hspace{0.2cm}
\ar@{^{(}->}[rr]^{\Phi_{\left\vert -\ell K_{X_{Q}}\right\vert }} & &
\mathbb{P}_{\mathbb{C}}^{\delta_{Q}}}
\end{equation*}
with 
\begin{equation*}
\mathbb{T}\ni t\longmapsto(\Phi_{\left\vert -\ell K_{X_{Q}}\right\vert
}\circ\iota)(t):=[...:z_{(i,j)}:...]_{(i,j)\in(\ell\mathring{Q})\cap \mathbb{%
Z}^{2}}\in\mathbb{P}_{\mathbb{C}}^{\delta_{Q}},\ z_{(i,j)}:=\chi^{(i,j)}(t),
\end{equation*}
where $\chi^{(i,j)}:\mathbb{T}\rightarrow\mathbb{C}^{\ast}$ is the character
associated to the lattice point $(i,j)$ (with $\mathbb{T}$ denoting the
algebraic torus Hom$_{\mathbb{Z}}(\mathbb{Z}^{2},\mathbb{C}^{\ast})$), for
all $(i,j)\in(\ell\mathring{Q})\cap\mathbb{Z}^{2}.$ The image $%
\Phi_{\left\vert -\ell K_{X_{Q}}\right\vert }(X_{Q})$ of $X_{Q}$ under $%
\Phi_{\left\vert -\ell K_{X_{Q}}\right\vert }$ is the Zariski closure of Im$%
(\Phi_{\left\vert -\ell K_{X_{Q}}\right\vert }\circ\iota)$ in $\mathbb{P}_{%
\mathbb{C}}^{\delta_{Q}}$ and can be viewed as the projective variety Proj$%
(S_{\ell\mathring{Q}}),$ where 
\begin{equation*}
S_{\ell\mathring{Q}}:=\mathbb{C}[C(\ell\mathring{Q})\cap\mathbb{Z}^{3}]={%
\displaystyle\bigoplus\limits_{\kappa=0}^{\infty}} \left( {\displaystyle%
\bigoplus\limits_{(i,j)\in(\kappa(\ell\mathring{Q}))\cap\mathbb{Z}^{2}}} 
\mathbb{C\cdot}\chi^{(i,j)}s^{\kappa}\right) \ 
\end{equation*}
$($with $C(\ell\mathring{Q}):=\{(\lambda y_{1},\lambda y_{2},\lambda
)\left\vert \lambda\in\mathbb{R}_{\geq0}\text{ and }(y_{1},y_{2})\in \ell%
\mathring{Q}\right. \})$ is the semigroup algebra which is naturally graded
by setting deg$(\chi^{(i,j)}s^{\kappa}):=\kappa.$ (For a detailed exposition
see \cite[Theorem 2.3.1, p. 75; Proposition 5.4.7, pp. 237-238; Theorem
5.4.8, pp. 239-240, and Theorem 7.1.13, pp. 325-326]{CLS}.) Equivalently, it
can be viewed as the zero set $\mathbb{V}(I_{\mathcal{A}_{Q}})\subset\mathbb{%
P}_{\mathbb{C}}^{\delta_{Q}}$ of the homogeneous ideal $I_{\mathcal{A}%
_{Q}}:= $ Ker$(\pi_{Q}),$ where%
\begin{equation*}
\mathcal{A}_{Q}:=\left\{ (i,j,1)\left\vert (i,j)\in(\ell\mathring{Q})\cap%
\mathbb{Z}^{2}\right. \right\} \subset\mathbb{Z}^{2}\times \{1\}\subset%
\mathbb{Z}^{3},
\end{equation*}
and $\pi_{Q}$ is the $\mathbb{C}$-algebra homomorphism 
\begin{equation*}
\mathbb{C}[...:z_{(i,j)}:....]_{(i,j)\in(\ell\mathring{Q})\cap\mathbb{Z}^{2}}%
\overset{\pi_{Q}}{\longrightarrow}\mathbb{C}[...,%
\chi^{(i,j,1)},....]_{(i,j,1)\in\mathcal{A}_{Q}},\ \
z_{(i,j)}\longmapsto\chi^{(i,j,1)}.
\end{equation*}
Furthermore, the \textit{projective degree} $d_{Q}:=$ deg$(\mathbb{V}(I_{%
\mathcal{A}_{Q}}))$ of $\mathbb{V}(I_{\mathcal{A}_{Q}})$ (i.e., the double
of the leading coefficient of the Hilbert polynomial of the homogeneous
coordinate ring $\mathbb{C}[...:z_{(i,j)}:....]_{(i,j)\in(\ell\mathring {Q}%
)\cap\mathbb{Z}^{2}}/I_{\mathcal{A}_{Q}}$) equals%
\begin{equation}
d_{Q}=2\,\text{area}(\ell\mathring{Q}).  \label{PROJDEGREE}
\end{equation}
(See Sturmfels \cite[Theorem 4.16, pp. 36-37, and p. 131]{Sturmfels} and 
\cite[Proposition 9.4.3, pp. 432-433]{CLS}.)

\begin{theorem}[Koelman \protect\cite{Koelman}]
\label{KOELMANTHM}If $\sharp(\partial(\ell\mathring {Q})\cap\mathbb{Z}%
^{2})\geq4,$ then $I_{\mathcal{A}_{Q}}$ is generated by all possible
quadratic binomials, i.e.,{\small 
\begin{equation*}
I_{\mathcal{A}_{Q}}=\left\langle \left\{
z_{(i_{1},j_{1})}z_{(i_{2},j_{2})}-z_{(i_{1}^{\prime},j_{1}^{%
\prime})}z_{(i_{2}^{\prime},j_{2}^{\prime})}\left\vert 
\begin{array}{c}
(i_{1},j_{1}),(i_{2},j_{2}),(i_{1}^{\prime},j_{1}^{\prime}),(i_{2}^{\prime
},j_{2}^{\prime})\in(\ell\mathring{Q})\cap\mathbb{Z}^{2},\medskip \\ 
\text{\emph{with} }(i_{1},j_{1})+(i_{2},j_{2})=(i_{1}^{\prime},j_{1}^{\prime
})+(i_{2}^{\prime},j_{2}^{\prime})%
\end{array}
\right. \right\} \right\rangle .
\end{equation*}
}
\end{theorem}

\begin{corollary}[{Castryck \& Cools \protect\cite[\S 2]{CC}}]
\label{CCCOR}If $\sharp(\partial (\ell\mathring{Q})\cap\mathbb{Z}^{2})\geq4,$
and if we denote by $\beta_{Q}$ the cardinality of any minimal system of
quadrics generating the ideal $I_{\mathcal{A}_{Q}},$ then%
\begin{equation}
\beta_{Q}=\tbinom{\delta_{Q}+2}{2}-\sharp(2(\ell\mathring{Q})\cap \mathbb{Z}%
^{2}).  \label{BETAQFORMULA}
\end{equation}
\end{corollary}

\begin{proof}
If HP$_{2}(\mathbb{P}_{\mathbb{C}}^{\delta_{Q}}):=\left\{ \text{homogeneous
polynomials (in }\delta_{Q}+1\text{ variables) of degree }2\right\} ,$ then
the $\mathbb{C}$-vector space homomorphism 
\begin{equation*}
f:\text{HP}_{2}(\mathbb{P}_{\mathbb{C}}^{\delta_{Q}})\longrightarrow \mathbb{%
C}\left[ x^{\pm1},y^{\pm1}\right] ,\text{ mapping }%
z_{(i_{1},j_{1})}z_{(i_{2},j_{2})}\text{ onto }%
x^{i_{1}+i_{2}}y^{j_{1}+j_{2}},
\end{equation*}
has as kernel Ker$(f)$ the $\mathbb{C}$-vector space of homogeneous
polynomials of degree $2$ which belong to $I_{\mathcal{A}_{Q}}$ and as image
Im$(f)$ the linear span of $\left. \{x^{i}y^{j}\right\vert (i,j)\in 2(\ell%
\mathring{Q})\cap\mathbb{Z}^{2}\}$ (because every lattice point in $2(\ell%
\mathring{Q})$ is the sum of two lattice points of $\ell\mathring{Q},$ cf. 
\cite[Theorem 2.2.12, pp. 68-69]{CLS}). Taking into account Koelman's
Theorem \ref{KOELMANTHM}, \cite[Lemma 4.1, p. 31]{Sturmfels}, and the fact
that $\mathbb{V}(I_{\mathcal{A}_{Q}})$ is not contained in any hyperplane of 
$\mathbb{P}_{\mathbb{C}}^{\delta_{Q}},$ the equality 
\begin{equation*}
\dim_{\mathbb{C}}(\text{Ker}(f))=\dim_{\mathbb{C}}(\text{HP}_{2}(\mathbb{P}_{%
\mathbb{C}}^{\delta_{Q}}))-\dim_{\mathbb{C}}(\func{Im}(f))
\end{equation*}
gives (\ref{BETAQFORMULA}).
\end{proof}

\noindent{}$\blacktriangleright$ \textit{Back to toric log del Pezzos with
one singularity}. Now let $Q$ be an LDP-polygon such that $X_{Q}$ has
exactly one singularity. According to Theorem \ref{MAIN1}, there exist $p\in 
\mathbb{Z}_{>0}$ and $k\in\{1,2,3\},$ such that $X_{Q}\cong X_{Q_{p}^{\left[
k\right] }}$ with index $\ell=\frac{p+1}{2}$ for $p$ odd and $\ell=p+1$ for $%
p$ even. For this reason, to apply Corollary \ref{CCCOR} and to prove
Theorem \ref{MAIN2} we shall take a closer look at the dilated polars $\ell 
\mathring{Q}_{p}^{\left[ k\right] }$ of the polygons $Q_{p}^{\left[ k\right]
}$ defined in (\ref{DEFQS}).

\begin{lemma}
\label{VERTPOLAR}The vertex sets of the polygons $\ell\mathring{Q}_{p}^{%
\left[ k\right] },$ $k\in\{1,2,3\},$ are the following\emph{:}%
\begin{equation*}
\begin{array}{l}
\mathcal{V}(\ell\mathring{Q}_{p}^{\left[ 1\right] })=\left\{ 
\begin{array}{ll}
\left\{ \tbinom{-1}{\frac{p-1}{2}},\tbinom{\frac{p+1}{2}}{-\frac{(p+1)^{2}}{2%
}},\tbinom{\frac{p+1}{2}}{p+1}\right\} , & \text{\emph{if} }p\text{ \emph{is
odd,}} \\ 
\  &  \\ 
\left\{ \tbinom{-2}{p-1},\tbinom{p+1}{-(p+1)^{2}},\tbinom{p+1}{2(p+1)}%
\right\} , & \text{\emph{if} }p\text{ \emph{is even,}}%
\end{array}
\right. \\ 
\  \\ 
\mathcal{V}(\ell\mathring{Q}_{p}^{\left[ 2\right] })=\left\{ 
\begin{array}{ll}
\left\{ \tbinom{-1}{\frac{p-1}{2}},\tbinom{0}{-\frac{p+1}{2}},\tbinom {\frac{%
p+1}{2}}{-\frac{p(p+1)}{2}},\tbinom{\frac{p+1}{2}}{p+1}\right\} , & \text{%
\emph{if} }p\text{ \emph{is odd,}} \\ 
\  &  \\ 
\left\{ \tbinom{-2}{p-1},\tbinom{0}{-(p+1)},\tbinom{p+1}{-p(p+1)},\tbinom{p+1%
}{2(p+1)}\right\} , & \text{\emph{if} }p\text{ \emph{is even,}}%
\end{array}
\right. \\ 
\  \\ 
\mathcal{V}(\ell\mathring{Q}_{p}^{\left[ 3\right] })=\left\{ 
\begin{array}{ll}
\left\{ \tbinom{-1}{\frac{p-1}{2}},\tbinom{0}{-\frac{p+1}{2}},\tbinom {\frac{%
p+1}{2}}{-\frac{p(p+1)}{2}},\tbinom{\frac{p+1}{2}}{\frac{p+1}{2}},\tbinom{0}{%
\frac{p+1}{2}}\right\} , & \text{\emph{if} }p\text{ \emph{is odd,}} \\ 
\  &  \\ 
\left\{ \tbinom{-2}{p-1},\tbinom{0}{-(p+1)},\tbinom{p+1}{-p(p+1)},\tbinom{p+1%
}{p+1},\tbinom{0}{p+1}\right\} , & \text{\emph{if} }p\text{ \emph{is even.}}%
\end{array}
\right.%
\end{array}%
\end{equation*}
\end{lemma}

\begin{proof}
Since $Q_{p}^{\left[ 1\right] }=$ conv$(\{\mathbf{v}_{1},\mathbf{v}_{2},%
\mathbf{v}_{3}\})$ with $\mathbf{v}_{1}=\tbinom{1}{-1},$ $\mathbf{v}_{2}=%
\tbinom{p}{1},\mathbf{v}_{3}=\tbinom{-1}{0},$ and 
\begin{equation*}
\boldsymbol{\eta}_{\text{conv}(\{\mathbf{v}_{1},\mathbf{v}_{2}\})}=\tbinom{%
\frac{2\ell}{p+1}}{-\frac{\left( p-1\right) \ell}{p+1}},\ \boldsymbol{\eta}_{%
\text{conv}(\{\mathbf{v}_{2},\mathbf{v}_{3}\})}=\tbinom{-1}{p+1},\ 
\boldsymbol{\eta}_{\text{conv}(\{\mathbf{v}_{3},\mathbf{v}_{1}\})}=\tbinom{-1%
}{-2},
\end{equation*}
with $l_{\text{conv}(\{\mathbf{v}_{1},\mathbf{v}_{2}\})}=\ell,$ $l_{\text{%
conv}(\{\mathbf{v}_{2},\mathbf{v}_{3}\})}=l_{\text{conv}(\{\mathbf{v}_{3},%
\mathbf{v}_{1}\})}=1,$ (\ref{VERTICESPOLAR}) gives%
\begin{equation*}
\mathcal{V}(\mathring{Q}_{p}^{\left[ 1\right] })=\left\{ \tbinom{-\frac {2}{%
p+1}}{\ \frac{p-1}{p+1}},\tbinom{1}{-(p+1)},\tbinom{1}{2}\right\} .
\end{equation*}
Analogously, we conclude that 
\begin{equation*}
\mathcal{V}(\mathring{Q}_{p}^{\left[ 2\right] })=\left\{ \tbinom{-\frac {2}{%
p+1}}{\ \frac{p-1}{p+1}},\tbinom{0}{-1},\tbinom{1}{-p},\tbinom{1}{2}\right\}
,\ \text{ }\mathcal{V}(\mathring{Q}_{p}^{\left[ 3\right] })=\left\{ \tbinom{-%
\frac{2}{p+1}}{\ \frac{p-1}{p+1}},\tbinom{0}{-1},\tbinom{1}{-p},\tbinom{1}{1}%
,\tbinom{0}{1}\right\} .
\end{equation*}
After multiplication with the index $\ell$ we get $\mathcal{V}(\ell 
\mathring{Q}_{p}^{\left[ k\right] }),$ $k\in\{1,2,3\}.$
\end{proof}

\begin{lemma}
\label{LEMMABOUNDARY}The number of lattice points on $\partial(\ell 
\mathring{Q}_{p}^{\left[ k\right] })$ is given in the table\emph{:}{%
\renewcommand{\arraystretch}{1.5} 
\begin{equation*}
\begin{tabular}{|c|c|c|c|c|c|c|c|}
\hline
\emph{No.} & $p$ & $k$ & $\sharp(\partial(\ell\mathring{Q}_{p}^{\left[ k%
\right] })\cap\mathbb{Z}^{2})$ & \emph{No.} & $p$ & $k$ & $\sharp
(\partial(\ell\mathring{Q}_{p}^{\left[ k\right] })\cap\mathbb{Z}^{2})$ \\ 
\hline\hline
\emph{(i)} & \emph{odd} & $1$ & $\frac{1}{2}\allowbreak\left( p+3\right)
^{2} $ & \emph{(iv)} & \emph{even} & $2$ & $p^{2}+5p+8$ \\ \hline
\emph{(ii)} & \emph{even} & $1$ & $(p+3)^{2}$ & \emph{(v)} & \emph{odd} & $3$
& $\frac{1}{2}\left( p^{2}+4p+7\right) $ \\ \hline
\emph{(iii)} & \emph{odd} & $2$ & $\frac{1}{2}\left( p^{2}+5p+8\right) $ & 
\emph{(vi)} & \emph{even} & $3$ & $p^{2}+4p+7$ \\ \hline
\end{tabular}
\ \ 
\end{equation*}
}
\end{lemma}

\begin{proof}
Since the number of lattice points lying on the boundary of a
lattice-polygon (w.r.t. $\mathbb{Z}^{2}$) is computed by the sum of the
greatest common divisors of the differences of the vertex-coordinates of its
edges, the above table is produced directly by using Lemma \ref{VERTPOLAR}.
\end{proof}

\begin{remark}
Since $\sharp(\partial(\ell\mathring{Q}_{p}^{\left[ k\right] })\cap\mathbb{Z}%
^{2})\geq6$ for all $p\in\mathbb{Z}_{>0}$ and all $k\in\{1,2,3\},$ Theorem %
\ref{KOELMANTHM} and Corollary \ref{CCCOR} can be applied for the
LDP-polygons $Q_{p}^{\left[ k\right] }.$
\end{remark}

\begin{lemma}
\label{DEGREES}The projective degree $d_{Q_{p}^{\left[ k\right] }}$ of $%
\mathbb{V}(I_{\mathcal{A}_{Q_{p}^{\left[ k\right] }}})$ is given in the table%
\emph{:} 
{\small 
\begin{equation*}
\begin{tabular}{|c|c|c|c|c|c|c|c|}
\hline
\emph{No.} & $p$ & $k$ & $%
\begin{array}{c}
d_{Q_{p}^{\left[ k\right]}} \vspace{0.1cm} \\ 
\end{array}%
$ & \emph{No.} & $p$ & $k$ & $d_{Q_{p}^{\left[ k\right] }}$ \\ \hline\hline
\emph{(i)} & \emph{odd} & $1$ & $\frac{1}{4}\allowbreak\left( p+1\right)
\left( p+3\right) ^{2}$ & \emph{(iv)} & \emph{even} & $2$ & $\left(
p+1\right) \left( p^{2}+5p+8\right) $ \\ \hline
\emph{(ii)} & \emph{even} & $1$ & $\allowbreak\left( p+1\right) \left(
p+3\right) ^{2}$ & \emph{(v)} & \emph{odd} & $3$ & $\frac{1}{4}\left(
p+1\right) \left( p^{2}+4p+7\right) $ \\ \hline
\emph{(iii)} & \emph{odd} & $2$ & $\frac{1}{4}\allowbreak\left( p+1\right)
\left( p^{2}+5p+8\right) $ & \emph{(vi)} & \emph{even} & $3$ & $\left(
p+1\right) \left( p^{2}+4p+7\right) $ \\ \hline
\end{tabular}%
\end{equation*}
}
\end{lemma}

\begin{proof}
To determine the area of $\ell\mathring{Q}_{p}^{\left[ k\right] }$ one may
work with its vertex set given in Lemma \ref{VERTPOLAR}. Alternatively,
using \cite[Proposition 2.10, p. 79]{Oda} and formula (\ref{SELFINTFORMULA})
for $X_{Q_{p}^{\left[ k\right] }}$ we deduce that%
\begin{equation*}
2\,\text{area}(\mathring{Q}_{p}^{\left[ k\right] })=K_{X_{Q_{p}^{\left[ k%
\right] }}}^{2}=6-k+p+\frac{4}{p+1},
\end{equation*}
and we read off $d_{Q_{p}^{\left[ k\right] }}$ easier via (\ref{PROJDEGREE})
which gives $d_{Q_{p}^{\left[ k\right] }}=\ell^{2}K_{X_{Q_{p}^{\left[ k%
\right] }}}^{2}$.
\end{proof}

\begin{lemma}
The dimension $\delta_{Q_{p}^{\left[ k\right] }}$ of the projective space in
which $\mathbb{V}(I_{\mathcal{A}_{Q_{p}^{\left[ k\right] }}})$ is embedded
equals%
\begin{equation}
\delta_{Q_{p}^{\left[ k\right] }}=\frac{1}{2}(d_{Q_{p}^{\left[ k\right]
}}+\sharp(\partial(\ell\mathring{Q}_{p}^{\left[ k\right] })\cap\mathbb{Z}%
^{2})).  \label{DELTAPICK}
\end{equation}
\end{lemma}

\begin{proof}
(\ref{DELTAPICK}) is immediate consequence of Pick's formula.
\end{proof}

\begin{lemma}
\label{BETACOMPUTATIONS}The number $\beta_{Q_{p}^{\left[ k\right] }}$ (of
the elements of any minimal generating system of\emph{\ }$I_{\mathcal{A}%
_{Q_{p}^{\left[ k\right] }}})$ is given by the formula:%
\begin{equation}
\beta_{Q_{p}^{\left[ k\right] }}=\tfrac{1}{2}(\delta_{Q_{p}^{\left[ k\right]
}}+1)(\delta_{Q_{p}^{\left[ k\right] }}+2)-(2d_{Q_{p}^{\left[ k\right]
}}+\sharp(\partial(\ell\mathring{Q}_{p}^{\left[ k\right] })\cap\mathbb{Z}%
^{2})+1).  \label{BETAQPKFORMULA}
\end{equation}
\end{lemma}

\begin{proof}
By the main properties of Ehrhart polynomial of the lattice polygon $\ell%
\mathring{Q}_{p}^{\left[ k\right] }$ (cf. \cite[Example 9.4.4, p. 433]{CLS})
we obtain%
\begin{equation*}
\sharp(2(\ell\mathring{Q}_{p}^{\left[ k\right] })\cap\mathbb{Z}^{2})=4\,%
\text{area}(\ell\mathring{Q})+\sharp(\partial(\ell\mathring{Q}_{p}^{\left[ k%
\right] })\cap\mathbb{Z}^{2})+1.
\end{equation*}
Hence, (\ref{BETAQPKFORMULA}) follows from (\ref{BETAQFORMULA}) and (\ref%
{PROJDEGREE}).
\end{proof}

\noindent Hyperplanes $\mathcal{H}\subset\mathbb{P}_{\mathbb{C}%
}^{\delta_{Q_{p}^{\left[ k\right] }}}$ give curves $\mathbb{V}(I_{\mathcal{A}%
_{Q_{p}^{\left[ k\right] }}})\cap\mathcal{H}$ which are linearly equivalent
to $-\ell K_{X_{Q_{p}^{\left[ k\right] }}}.$ For \textit{generic} $\mathcal{H%
}$'s the intersection $\mathcal{C}_{Q_{p}^{\left[ k\right] }}:=\mathbb{V}(I_{%
\mathcal{A}_{Q_{p}^{\left[ k\right] }}})\cap\mathcal{H}$ is (by Bertini's
Theorem) a smooth connected curve in the smooth locus of $\mathbb{V}(I_{%
\mathcal{A}_{Q_{p}^{\left[ k\right] }}})\cong X_{Q_{p}^{\left[ k\right] }}.$
The genus of $\mathcal{C}_{Q_{p}^{\left[ k\right] }}$ is called the\textit{\
sectional genus} $g_{Q_{p}^{\left[ k\right] }}\emph{\medskip}$ \textit{of} $%
X_{Q_{p}^{\left[ k\right] }}.$

\begin{lemma}
The sectional genus of $X_{Q_{p}^{\left[ k\right] }}$ is%
\begin{equation}
g_{Q_{p}^{\left[ k\right] }}=\delta_{Q_{p}^{\left[ k\right]
}}-\sharp(\partial(\ell\mathring{Q}_{p}^{\left[ k\right] })\cap\mathbb{Z}%
^{2})+1.  \label{GQPKFORMULA}
\end{equation}
\end{lemma}

\begin{proof}
(\ref{GQPKFORMULA}) follows from the fact that $g_{Q_{p}^{\left[ k\right]
}}=\sharp($int$(\ell\mathring{Q}_{p}^{\left[ k\right] })\cap\mathbb{Z}^{2}).$
(See \cite[Proposition 10.5.8, p. 509]{CLS}.)
\end{proof}

\noindent\textit{Proof of Theorem }\ref{MAIN2}: The number $\sharp
(\partial(\ell\mathring{Q}_{p}^{\left[ k\right] })\cap\mathbb{Z}^{2})$ and
the projective degree $d_{Q_{p}^{\left[ k\right] }}$ are known from Lemmas %
\ref{LEMMABOUNDARY} and \ref{DEGREES}, respectively, while $\delta _{Q_{p}^{%
\left[ k\right] }}$ is computed via (\ref{DELTAPICK}), leading to Table (\ref%
{TABELLE1}), and consequently to Table (\ref{TABELLE2}) by making use of
formula (\ref{BETAQPKFORMULA}). Finally, one obtains Table (\ref{TABELLE3})
by means of the equality (\ref{GQPKFORMULA}). \hfill$\square$

\begin{note}
For a \texttt{Magma} code for the computation of a minimal generating system
of the ideal defining the projective toric surface associated to an \textit{%
arbitrary} lattice polygon, see \cite{CC2}. In our particular case (in which
we deal only with \textit{quadrics}) it is enough to collect all vectorial
relations $(i_{1},j_{1})+(i_{2},j_{2})=(i_{1}^{\prime},j_{1}^{%
\prime})+(i_{2}^{\prime},j_{2}^{\prime}),$ and to determine a $\mathbb{C}$%
-linearly independent subset of the set of the corresponding quadratic
binomials $z_{(i_{1},j_{1})}z_{(i_{2},j_{2})}-z_{(i_{1}^{\prime},j_{1}^{%
\prime})}z_{(i_{2}^{\prime},j_{2}^{\prime})}$ by simply applying Gaussian
elimination. For a short routine (written in \texttt{Python}) see \cite%
{Dais-Markakis}.
\end{note}

\begin{examples}
(i) The ideal $I_{\mathcal{A}_{Q_{1}^{\left[ 2\right] }}}$ (with $\mathbb{V}%
(I_{\mathcal{A}_{Q_{1}^{\left[ 2\right] }}})\subset \mathbb{P}_{\mathbb{C}%
}^{7}$) is minimally generated by the following $14$ quadrics: {\small 
\begin{equation*}
\begin{array}{l}
z_{(-1,0)}z_{(1,-1)}-z_{(0,-1)}z_{(0,0)},\
z_{(-1,0)}z_{(1,0)}-z_{(0,-1)}z_{(0,1)},\ z_{(1,0)}^{2}-z_{(1,1)}z_{(1,-1)},
\\ 
z_{(-1,0)}z_{(1,1)}-z_{(0,0)}z_{(0,1)},\
z_{(1,1)}z_{(1,0)}-z_{(1,2)}z_{(1,-1)},\ z_{(1,1)}^{2}-z_{(1,2)}z_{(1,0)},
\\ 
z_{(0,1)}^{2}-z_{(-1,0)}z_{(1,2)},\
z_{(0,1)}z_{(1,-1)}-z_{(0,-1)}z_{(1,1),}\
z_{(0,1)}z_{(1,0)}-z_{(0,-1)}z_{(1,2)}, \\ 
z_{(0,1)}z_{(1,1)}-z_{(0,0)}z_{(1,2)},\
z_{(0,0)}z_{(1,-1)}-z_{(0,-1)}z_{(1,0)},\
z_{(0,0)}z_{(1,0)}-z_{(0,-1)}z_{(1,1)}, \\ 
z_{(0,0)}z_{(1,1)}-z_{(0,-1)}z_{(1,2)},\ z_{(0,0)}^{2}-z_{(0,-1)}z_{(0,1)}.%
\end{array}%
\end{equation*}
} (ii) Correspondingly, the $9$ quadrics {\small 
\begin{equation*}
\begin{array}{l}
z_{(-1,0)}z_{(1,0)}-z_{(0,1)}z_{(0,-1)},\
z_{(1,0)}^{2}-z_{(1,1)}z_{(1,-1)},\ z_{(-1,0)}z_{(1,-1)}-z_{(0,0)}z_{(0,-1)},
\\ 
z_{(-1,0)}z_{(1,1)}-z_{(0,1)}z_{(0,0)},\
z_{(0,-1)}z_{(1,0)}-z_{(0,0)}z_{(1,-1)},\
z_{(0,-1)}z_{(1,1)}-z_{(0,1)}z_{(1,-1)}, \\ 
z_{(0,0)}z_{(1,0)}-z_{(0,1)}z_{(1,-1)},\
z_{(0,0)}z_{(1,1)}-z_{(0,1)}z_{(1,0)},\ z_{(0,0)}^{2}-z_{(0,1)}z_{(0,-1)}%
\end{array}%
\end{equation*}
build a }minimal set of generators of the ideal $I_{\mathcal{A}_{Q_{1}^{%
\left[ 3\right] }}},$ and $\mathbb{V}(I_{\mathcal{A}_{Q_{1}^{\left[ 3\right]
}}})\subset\mathbb{P}_{\mathbb{C}}^{6}.$ ($X_{Q_{1}^{\left[ 3\right] }}$ is
obtained by blowing up $X_{Q_{1}^{\left[ 2\right] }}$ at one non-singular
point, cf. Figure \ref{Fig.5}.)
\end{examples}

\begin{figure}[ht]
\includegraphics[height=4cm, width=11.5cm]{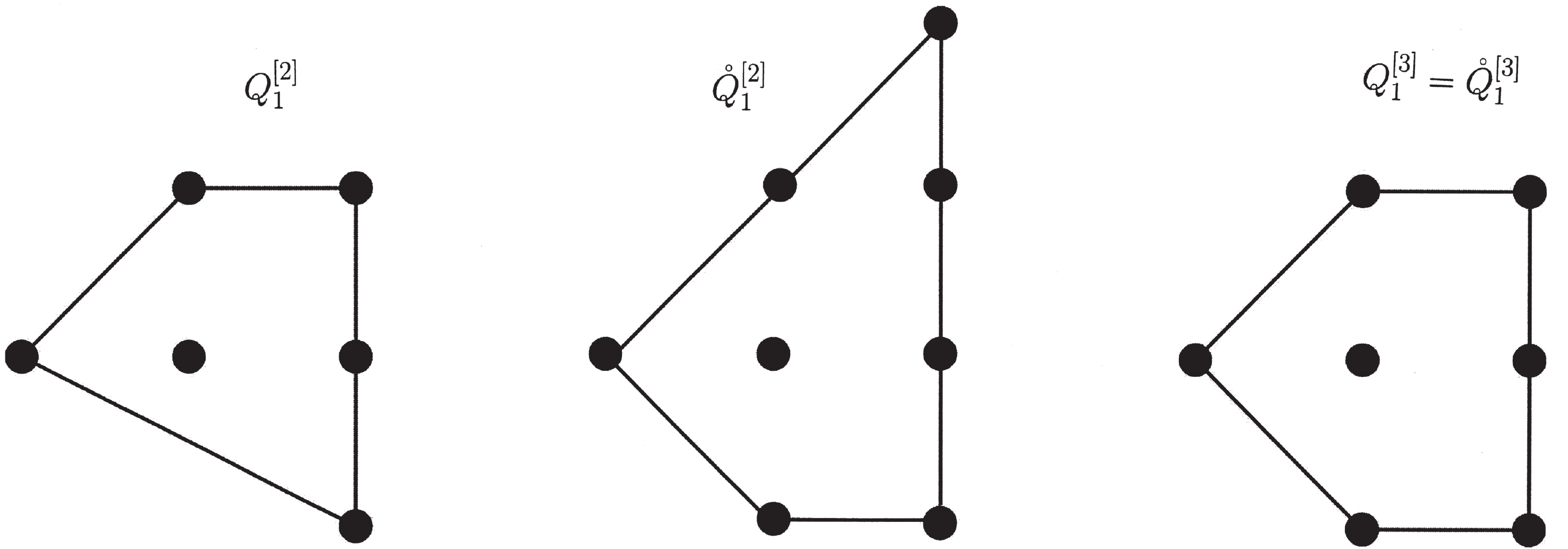}
\caption{}
\label{Fig.5}
\end{figure}

\noindent(iii) The next coming example, namely that one created by the
LDP-polygon $Q_{3}^{\left[ 3\right] }$ (cf. Figure \ref{Fig.6}), in which $2%
\mathring{Q}_{3}^{\left[ 3\right] }\cap\mathbb{Z}^{2}$ consists of $22$
lattice points and $\mathbb{V}(I_{\mathcal{A}_{Q_{3}^{\left[ 3\right]
}}})\subset\mathbb{P}_{\mathbb{C}}^{21},$ is much more complicated$.$ Using 
\cite{Dais-Markakis} we see that $I_{\mathcal{A}_{Q_{3}^{\left[ 3\right] }}}$
is minimally generated by the following $182$ quadrics:

{\scriptsize 
\begin{equation*}
\begin{array}{lll}
z_{(0,-2)}z_{(2,-2)}-z_{(1,-4)}z_{(1,0)}, & 
z_{(1,-4)}z_{(2,-4)}-z_{(1,-2)}z_{(2,-6)},\newline
& z_{(-1,1)}z_{(1,-1)}-z_{(0,-2)}z_{(0,2)}, \\ 
z_{(0,-2)}z_{(2,0)}-z_{(1,-1)}z_{(1,-1)}, & 
z_{(-1,1)}z_{(1,-4)}-z_{(0,-2)}z_{(0,-1)}, & 
z_{(1,-4)}z_{(2,-1)}-z_{(1,-1)}z_{(2,-4)}, \\ 
z_{(0,-2)}z_{(2,-4)}-z_{(1,-4)}z_{(1,-2)}, & 
z_{(2,-6)}z_{(2,-1)}-z_{(2,-5)}z_{(2,-2)}, & 
z_{(0,-2)}z_{(2,-3)}-z_{(0,0)}z_{(2,-5)}, \\ 
z_{(-1,1)}z_{(2,0)}-z_{(0,2)}z_{(1,-1)}, & 
z_{(0,-1)}z_{(2,2)}-z_{(0,0)}z_{(2,1)}, & 
z_{(-1,1)}z_{(2,-6)}-z_{(0,-1)}z_{(1,-4)}, \\ 
z_{(0,-2)}z_{(2,-5)}-z_{(0,-1)}z_{(2,-6)}, & 
z_{(1,-4)}z_{(2,2)}-z_{(1,-1)}z_{(2,-1)}, & 
z_{(1,-4)}z_{(2,1)}-z_{(1,0)}z_{(2,-3)}, \\ 
z_{(-1,1)}z_{(2,-5)}-z_{(0,-2)}z_{(1,-2)}, & 
z_{(-1,1)}z_{(2,1)}-z_{(0,2)}z_{(1,0)}, & 
z_{(-1,1)}z_{(2,-3)}-z_{(0,1)}z_{(1,-3)}, \\ 
z_{(-1,1)}z_{(2,-5)}-z_{(0,-1)}z_{(1,-3)}, & 
z_{(1,-4)}z_{(2,0)}-z_{(1,-2)}z_{(2,-2)}, & 
z_{(1,-4)}z_{(2,1)}-z_{(1,-1)}z_{(2,-2)}, \\ 
z_{(1,-4)}z_{(2,-3)}-z_{(1,-1)}z_{(2,-6)}, & 
z_{(1,-4)}z_{(2,2)}-z_{(1,-3)}z_{(2,1)}, & 
z_{(-1,1)}z_{(2,-1)}-z_{(0,1)}z_{(1,-1)}, \\ 
z_{(0,0)}z_{(2,2)}-z_{(0,2)}z_{(2,0)}, & 
z_{(-1,1)}z_{(2,-4)}-z_{(0,1)}z_{(1,-4)}, & 
z_{(1,-3)}z_{(2,2)}-z_{(1,0)}z_{(2,-1)}, \\ 
z_{(0,-2)}z_{(2,-2)}-z_{(1,-2)}^{2}, & 
z_{(1,-4)}z_{(2,0)}-z_{(1,1)}z_{(2,-5)}, & 
z_{(1,-4)}z_{(2,-2)}-z_{(1,0)}z_{(2,-6)}, \\ 
z_{(2,-6)}z_{(2,-1)}-z_{(2,-4)}z_{(2,-3)}, & 
z_{(2,-6)}z_{(2,2)}-z_{(2,-4)}z_{(2,0)}, & z_{(2,-2)}z_{(2,2)}-z_{(2,0)}^{2},
\\ 
z_{(1,-3)}z_{(2,2)}-z_{(1,-2)}z_{(2,1)}, & 
z_{(0,0)}z_{(2,2)}-z_{(0,1)}z_{(2,1)}, & 
z_{(1,-1)}z_{(2,2)}-z_{(1,2)}z_{(2,-1)}, \\ 
z_{(-1,1)}z_{(2,-2)}-z_{(0,-2)}z_{(1,1)}, & 
z_{(2,-6)}z_{(2,-2)}-z_{(2,-4)}z_{(2,-4)}, & 
z_{(1,-1)}z_{(2,2)}-z_{(1,1)}z_{(2,0)}, \\ 
z_{(-1,1)}z_{(2,-4)}-z_{(0,-1)}z_{(1,-2)}, & 
z_{(1,-4)}z_{(2,0)}-z_{(1,-3)}z_{(2,-1)}, & 
z_{(0,-2)}z_{(2,-3)}-z_{(1,-3)}z_{(1,-2)}, \\ 
z_{(0,-2)}z_{(2,2)}-z_{(0,1)}z_{(2,-1)}, & 
z_{(1,-4)}z_{(2,2)}-z_{(1,1)}z_{(2,-3)}, & 
z_{(1,1)}z_{(2,2)}-z_{(1,2)}z_{(2,1)}, \\ 
z_{(2,-4)}z_{(2,2)}-z_{(2,-3)}z_{(2,1)}, & 
z_{(2,-6)}z_{(2,2)}-z_{(2,-3)}z_{(2,-1)}, & 
z_{(2,-4)}z_{(2,2)}-z_{(2,-2)}z_{(2,0)}, \\ 
z_{(-1,1)}z_{(2,-1)}-z_{(0,0)}z_{(1,0)}, & 
z_{(-1,1)}z_{(2,0)}-z_{(0,-1)}z_{(1,2)}, & 
z_{(1,-4)}z_{(2,0)}-z_{(1,0)}z_{(2,-4)}, \\ 
z_{(-1,1)}z_{(2,-4)}-z_{(0,0)}z_{(1,-3)}, & 
z_{(0,-2)}z_{(2,-4)}-z_{(0,-1)}z_{(2,-5)}, & 
z_{(0,-2)}z_{(2,1)}-z_{(1,-2)}z_{(1,1)}, \\ 
z_{(-1,1)}z_{(2,0)}-z_{(0,0)}z_{(1,1)}, & 
z_{(2,-6)}z_{(2,1)}-z_{(2,-5)}z_{(2,0)}, & 
z_{(-1,1)}z_{(2,-2)}-z_{(0,2)}z_{(1,-3)}, \\ 
z_{(0,-2)}z_{(2,-1)}-z_{(1,-3)}z_{(1,0)}, & 
z_{(-1,1)}z_{(1,-2)}-z_{(0,-1)}z_{(0,0)}, & 
z_{(1,-2)}z_{(2,2)}-z_{(1,0)}z_{(2,0)}, \\ 
z_{(0,1)}z_{(2,2)}-z_{(0,2)}z_{(2,1)}, & 
z_{(1,-4)}z_{(2,2)}-z_{(1,2)}z_{(2,-4)}, & 
z_{(2,-2)}z_{(2,2)}-z_{(2,-1)}z_{(2,1)}, \\ 
z_{(1,-4)}z_{(2,2)}-z_{(1,0)}z_{(2,-2)}, & 
z_{(0,-2)}z_{(2,0)}-z_{(0,1)}z_{(2,-3)}, & 
z_{(0,-2)}z_{(2,-4)}-z_{(1,-3)}^{2}, \\ 
z_{(1,-4)}z_{(2,-1)}-z_{(1,1)}z_{(2,-6)}, & 
z_{(1,-4)}z_{(2,-3)}-z_{(1,-3)}z_{(2,-4)}, & 
z_{(0,-2)}z_{(2,0)}-z_{(1,-4)}z_{(1,2)}, \\ 
z_{(0,-2)}z_{(2,0)}-z_{(1,-2)}z_{(1,0)}, & 
z_{(0,-2)}z_{(2,1)}-z_{(0,-1)}z_{(2,0)}, & 
z_{(1,-4)}z_{(2,-2)}-z_{(1,-1)}z_{(2,-5)}, \\ 
z_{(0,-2)}z_{(2,1)}-z_{(1,-1)}z_{(1,0)}, & 
z_{(-1,1)}z_{(2,-3)}-z_{(0,0)}z_{(1,-2)}, & 
z_{(-1,1)}z_{(2,-2)}-z_{(0,0)}z_{(1,-1)}, \\ 
z_{(0,-2)}z_{(2,-1)}-z_{(1,-2)}z_{(1,-1)}, & 
z_{(2,-6)}z_{(2,-4)}-z_{(2,-5)}^{2}, & 
z_{(-1,1)}z_{(2,-4)}-z_{(0,-2)}z_{(1,-1)}, \\ 
z_{(0,-2)}z_{(2,-1)}-z_{(1,-4)}z_{(1,1)}, & 
z_{(-1,1)}z_{(2,-3)}-z_{(0,-1)}z_{(1,-1)}, & 
z_{(0,-2)}z_{(2,2)}-z_{(0,-1)}z_{(2,1)}, \\ 
z_{(1,-4)}z_{(2,-1)}-z_{(1,-3)}z_{(2,-2)}, & 
z_{(-1,1)}z_{(1,0)}-z_{(0,-1)}z_{(0,2)}, & 
z_{(0,-2)}z_{(2,-2)}-z_{(1,-3)}z_{(1,-1)}, \\ 
z_{(-1,1)}z_{(1,1)}-z_{(0,1)}^{2}, & z_{(2,-6)}z_{(2,0)}-z_{(2,-3)}^{2}, & 
z_{(-1,1)}z_{(2,-1)}-z_{(0,-1)}z_{(1,1)}, \\ 
z_{(1,-2)}z_{(2,2)}-z_{(1,1)}z_{(2,-1)}, & 
z_{(0,-2)}z_{(2,1)}-z_{(1,-3)}z_{(1,2)}, & 
z_{(2,-5)}z_{(2,2)}-z_{(2,-2)}z_{(2,-1)}, \\ 
z_{(1,-4)}z_{(2,-4)}-z_{(1,-3)}z_{(2,-5)}, & 
z_{(1,-4)}z_{(2,-1)}-z_{(1,0)}z_{(2,-5)}, & 
z_{(1,-4)}z_{(2,1)}-z_{(1,1)}z_{(2,-4)}, \\ 
z_{(1,-4)}z_{(2,1)}-z_{(1,-2)}z_{(2,-1)}, & 
z_{(0,-2)}z_{(2,-1)}-z_{(0,0)}z_{(2,-3)}, & 
z_{(0,-2)}z_{(2,2)}-z_{(1,-2)}z_{(1,2)}, \\ 
z_{(1,-3)}z_{(2,2)}-z_{(1,-1)}z_{(2,0)}, & 
z_{(1,-4)}z_{(2,-2)}-z_{(1,-3)}z_{(2,-3)}, & 
z_{(0,-1)}z_{(2,2)}-z_{(1,0)}z_{(1,1)}, \\ 
z_{(2,-1)}z_{(2,2)}-z_{(2,0)}z_{(2,1)}, & 
z_{(0,-1)}z_{(2,2)}-z_{(0,1)}z_{(2,0)}, & 
z_{(0,-2)}z_{(2,-1)}-z_{(0,2)}z_{(2,-5)}, \\ 
z_{(1,-4)}z_{(2,-5)}-z_{(1,-3)}z_{(2,-6)}, & 
z_{(0,-2)}z_{(2,2)}-z_{(0,0)}z_{(2,0)}, & 
z_{(-1,1)}z_{(1,1)}-z_{(0,0)}z_{(0,2)}, \\ 
z_{(2,-6)}z_{(2,-2)}-z_{(2,-5)}z_{(2,-3)}, & 
z_{(2,-6)}z_{(2,1)}-z_{(2,-3)}z_{(2,-2)}, & 
z_{(-1,1)}z_{(2,-3)}-z_{(0,2)}z_{(1,-4)}, \\ 
z_{(1,-3)}z_{(2,2)}-z_{(1,2)}z_{(2,-3)}, & 
z_{(-1,1)}z_{(2,-2)}-z_{(0,-1)}z_{(1,0)}, & 
z_{(2,-5)}z_{(2,2)}-z_{(2,-4)}z_{(2,1)}, \\ 
z_{(-1,1)}z_{(1,2)}-z_{(0,1)}z_{(0,2)}, & 
z_{(0,-2)}z_{(2,-3)}-z_{(0,-1)}z_{(2,-4)}, & 
z_{(0,-2)}z_{(2,2)}-z_{(1,0)}^{2}, \\ 
z_{(-1,1)}z_{(1,-1)}-z_{(0,0)}^{2}, & 
z_{(1,-4)}z_{(2,0)}-z_{(1,-1)}z_{(2,-3)}, & 
z_{(-1,1)}z_{(2,2)}-z_{(0,2)}z_{(1,1)}, \\ 
z_{(1,-3)}z_{(2,2)}-z_{(1,1)}z_{(2,-2)}, & 
z_{(0,-2)}z_{(2,2)}-z_{(1,-1)}z_{(1,1)}, & 
z_{(0,-2)}z_{(2,-2)}-z_{(0,0)}z_{(2,-4)}, \\ 
z_{(-1,1)}z_{(1,-2)}-z_{(0,-2)}z_{(0,1)}, & 
z_{(0,-2)}z_{(2,0)}-z_{(0,0)}z_{(2,-2)}, & 
z_{(0,-2)}z_{(2,-2)}-z_{(0,2)}z_{(2,-6)}, \\ 
&  & 
\end{array}%
\end{equation*}%
\begin{equation*}
\begin{array}{lll}
z_{(0,1)}z_{(2,2)}-z_{(1,1)}z_{(1,2)}, & 
z_{(-1,1)}z_{(1,-3)}-z_{(0,-2)}z_{(0,0)}, & 
z_{(0,-2)}z_{(2,-1)}-z_{(0,-1)}z_{(2,-2)}, \\ 
z_{(0,-1)}z_{(2,2)}-z_{(1,-1)}z_{(1,2)}, & 
z_{(1,-4)}z_{(2,1)}-z_{(1,-3)}z_{(2,0)}, & 
z_{(-1,1)}z_{(1,-3)}-z_{(0,-1)}^{2}, \\ 
z_{(0,-2)}z_{(2,0)}-z_{(0,-1)}z_{(2,-1)}, & 
z_{(1,-4)}z_{(2,2)}-z_{(1,-2)}z_{(2,0)}, & 
z_{(-1,1)}z_{(2,-1)}-z_{(0,-2)}z_{(1,2)}, \\ 
z_{(0,-1)}z_{(2,2)}-z_{(0,2)}z_{(2,-1)}, & 
z_{(2,-6)}z_{(2,0)}-z_{(2,-4)}z_{(2,-2)}, & 
z_{(0,-2)}z_{(2,-2)}-z_{(0,1)}z_{(2,-5)}, \\ 
z_{(2,-3)}z_{(2,2)}-z_{(2,-1)}z_{(2,0)}, & 
z_{(0,0)}z_{(2,2)}-z_{(1,0)}z_{(1,2)}, & 
z_{(-1,1)}z_{(2,-5)}-z_{(0,0)}z_{(1,-4)}, \\ 
z_{(-1,1)}z_{(2,-6)}-z_{(0,-2)}z_{(1,-3)}, & 
z_{(-1,1)}z_{(2,-2)}-z_{(0,1)}z_{(1,-2)}, & 
z_{(0,-2)}z_{(2,-2)}-z_{(0,-1)}z_{(2,-3)}, \\ 
z_{(-1,1)}z_{(2,1)}-z_{(0,0)}z_{(1,2)}, & 
z_{(2,-3)}z_{(2,2)}-z_{(2,-2)}z_{(2,1)}, & 
z_{(-1,1)}z_{(2,1)}-z_{(0,1)}z_{(1,1)}, \\ 
z_{(1,0)}z_{(2,2)}-z_{(1,2)}z_{(2,0)}, & 
z_{(0,-2)}z_{(2,-1)}-z_{(0,1)}z_{(2,-4)}, & 
z_{(0,-2)}z_{(2,-5)}-z_{(1,-4)}z_{(1,-3)}, \\ 
z_{(-1,1)}z_{(2,-3)}-z_{(0,-2)}z_{(1,0)}, & 
z_{(1,-2)}z_{(2,2)}-z_{(1,2)}z_{(2,-2)}, & 
z_{(2,-5)}z_{(2,2)}-z_{(2,-3)}z_{(2,0)}, \\ 
z_{(-1,1)}z_{(2,0)}-z_{(0,1)}^{2}, & z_{(0,-2)}z_{(2,-6)}-z_{(1,-4)}^{2}, & 
z_{(0,-2)}z_{(2,2)}-z_{(0,2)}z_{(2,-2)}, \\ 
z_{(1,-4)}z_{(2,-2)}-z_{(1,-2)}z_{(2,-4)}, & 
z_{(-1,1)}z_{(2,2)}-z_{(0,1)}z_{(1,2)}, & 
z_{(1,-4)}z_{(2,1)}-z_{(1,2)}z_{(2,-5)}, \\ 
z_{(1,-4)}z_{(2,-3)}-z_{(1,-2)}z_{(2,-5)}, & 
z_{(-1,1)}z_{(1,0)}-z_{(0,0)}z_{(0,1)}, & 
z_{(1,0)}z_{(2,2)}-z_{(1,1)}z_{(2,1)}, \\ 
z_{(2,-4)}z_{(2,2)}-z_{(2,-1)}^{2}, & 
z_{(2,-6)}z_{(2,-3)}-z_{(2,-5)}z_{(2,-4)}, & 
z_{(2,-6)}z_{(2,2)}-z_{(2,-5)}z_{(2,1)}, \\ 
z_{(0,-2)}z_{(2,-3)}-z_{(1,-4)}z_{(1,-1)}, & 
z_{(2,0)}z_{(2,2)}-z_{(2,1)}^{2}, & z_{(2,-6)}z_{(2,1)}-z_{(2,-4)}z_{(2,-1)},
\\ 
z_{(0,2)}z_{(2,2)}-z_{(1,2)}^{2}, & z_{(0,-2)}z_{(2,1)}-z_{(0,2)}z_{(2,-3)},
& z_{(0,-2)}z_{(2,1)}-z_{(0,1)}z_{(2,-2)}, \\ 
z_{(-1,1)}z_{(1,-1)}-z_{(0,-1)}z_{(0,1)}, & 
z_{(0,-2)}z_{(2,-4)}-z_{(0,0)}z_{(2,-6)}, & 
z_{(1,-4)}z_{(2,-1)}-z_{(1,-2)}z_{(2,-3)}, \\ 
z_{(2,-6)}z_{(2,2)}-z_{(2,-2)}^{2}, & 
z_{(0,-2)}z_{(2,-3)}-z_{(0,1)}z_{(2,-6)}, & 
z_{(0,-2)}z_{(2,1)}-z_{(0,0)}z_{(2,-1)}, \\ 
z_{(0,-2)}z_{(2,0)}-z_{(1,-3)}z_{(1,1)}, & 
z_{(-1,1)}z_{(2,-1)}-z_{(0,2)}z_{(1,-2)}, & 
z_{(2,-6)}z_{(2,0)}-z_{(2,-5)}z_{(2,-1)}, \\ 
z_{(0,-2)}z_{(2,0)}-z_{(0,2)}z_{(2,-4)}, & 
z_{(0,0)}z_{(2,2)}-z_{(1,1)}z_{(1,1)}, & 
z_{(1,-4)}z_{(2,0)}-z_{(1,2)}z_{(2,-6)}, \\ 
z_{(1,-1)}z_{(2,2)}-z_{(1,0)}z_{(2,1)}, & 
z_{(1,-2)}z_{(2,2)}-z_{(1,-1)}z_{(2,1)}. & 
\end{array}%
\end{equation*}
}

\begin{figure}[h]
\includegraphics[height=7.8cm, width=10cm]{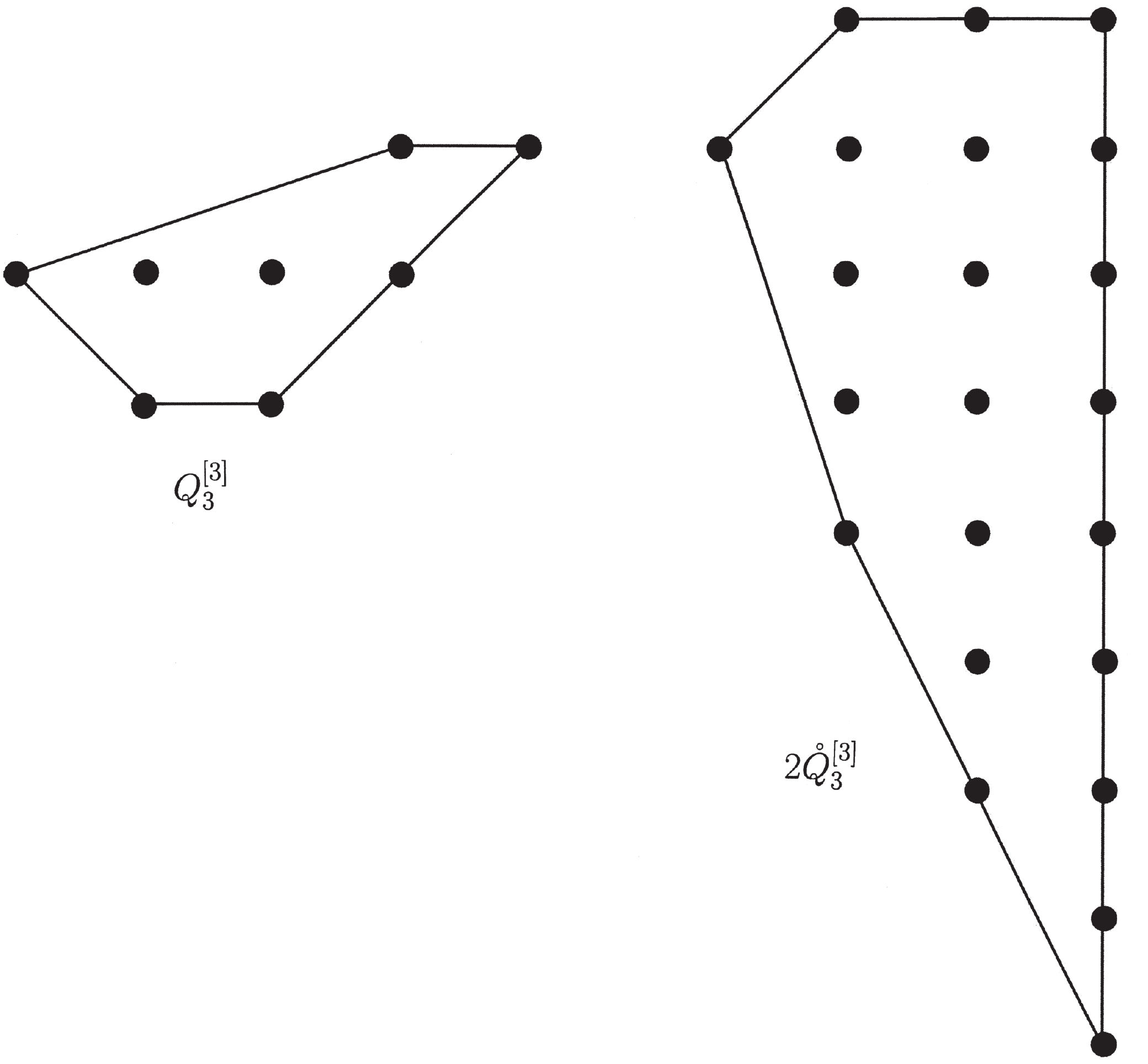}
\caption{}
\label{Fig.6}
\end{figure}

\end{document}